\documentclass[11pt]{amsart}
\usepackage{amsmath}
\usepackage{tikz}
 
\theoremstyle{plain}
\newtheorem{theorem}{Theorem}[section]
\newtheorem{lemma}[theorem]{Lemma}

\newtheorem{cor}[theorem]{Corollary}

\theoremstyle{definition}
\newtheorem{remark}[theorem]{Remark}

\numberwithin{equation}{section}
\def\be{\begin{equation}}
\def\ee{\end{equation}}

\begin{document}

\title[Convergence of Boundary Expansions]
{The Convergence of Boundary Expansions\\ and the Analyticity of Minimal Surfaces\\ in the Hyperbolic Space}
\author[Qing Han]{Qing Han}
\address{Department of Mathematics\\
University of Notre Dame\\
Notre Dame, IN 46556} \email{qhan@nd.edu}
\address{Beijing International Center for Mathematical Research\\
Peking University\\
Beijing, 100871, China} \email{qhan@math.pku.edu.cn}

%%%%%%%%%%%%%%%%%%%%%%%%%%%%%%%%%%
%%%%%%%%%%%%%%%%%%%%%%%%%%%%%%%%%%Changed my contact info
\author[Xumin Jiang]{Xumin Jiang}
\address{Department of Mathematics\\
Rutgers University\\
New Brunswick, NJ 08901}  \email{xj60@math.rutgers.edu}
%%%%%%%%%%%%%%%%%%%%%%%%%%%%%%%%%%
%%%%%%%%%%%%%%%%%%%%%%%%%%%%%%%%%%

\begin{abstract}
We study expansions near the boundary of solutions to the 
Dirichlet problem for minimal graphs in the hyperbolic space and prove the local
convergence of such expansions if the boundary is locally analytic. 
As a consequence, we prove a conjecture by F.-H. Lin that the minimal graph
is analytic up to the boundary if the boundary is analytic and the minimal graph is smooth
up to the boundary. 
\end{abstract}

\thanks{The first author acknowledges the support of NSF
Grant DMS-1404596. }
%\date{\today}
\maketitle

\section{Introduction}\label{sec-Intro}

Complete minimal hypersurfaces in the hyperbolic space
$\mathbb H^{n+1}$ demonstrate similar properties as
those in the Euclidean space $\mathbb R^{n+1}$ in the aspect of the interior regularity and 
different properties in the aspect of the boundary regularity.  
Anderson \cite{Anderson1982Invent}, \cite{Anderson1983} 
studied complete area-minimizing submanifolds 
and proved that, for any 
given closed embedded $(n-1)$-dimensional submanifold 
$N$ at the infinity of $\mathbb H^{n+1}$, 
there exists 
a complete 
area minimizing integral $n$-current which is asymptotic to
$N$ at infinity. 
In the case $n\le 6$, these currents are 
embedded smooth submanifolds; 
while in the case $n\ge 7$, as in the Euclidean case, there can be closed singular set 
of Hausdorff dimension at most $n-7$. 
Hardt and 
Lin \cite{Hardt&Lin1987}  discussed the $C^1$-boundary regularity of such hypersurfaces. 
Subsequently, 
Lin \cite{Lin1989Invent} studied the higher order boundary regularity. 
Recently, we \cite{HanJiang2014} studied the boundary 
expansions of the minimal graphs  in the hyperbolic space and established optimal asymptotic expansions
in the context of the finite regularity. In this paper, we will study the convergence of such 
expansions and the analyticity of minimal graphs up to the boundary. 

Assume $\Omega$ is a bounded domain in $\mathbb{R}^n$. 
Lin \cite{Lin1989Invent} studied the 
Dirichlet problem of the form
\be\label{Eqf}
\Delta f - \frac{f_i f_j}{1+|D f|^2}f_{ij}+\frac{n}{f} =0  \quad \text{in } \Omega, \ee
with the condition 
\begin{align}\label{EqfCondition}\begin{split}f &>0   \quad \text{in } \Omega, \\
f &=0 \quad \text{on } \partial \Omega.
\end{split}\end{align}
%In this paper, we follow Lin by denoting solutions of \eqref{Eqf} by $f$. 
We note that the equation \eqref{Eqf}
becomes singular on $\partial\Omega$ since $f=0$ there. 
If $\Omega$ is a $C^2$-domain in $\mathbb R^n$ with 
a nonnegative boundary mean curvature  
$H_{\partial\Omega}\ge 0$ with respect to the inward normal of 
$\partial\Omega$, then \eqref{Eqf} and \eqref{EqfCondition}
admit a unique solution $f\in C(\bar\Omega)\cap C^\infty(\Omega)$. 
Moreover,  the graph of $f$ 
is a complete minimal hypersurface in the hyperbolic space $\mathbb H^{n+1}$
with the asymptotic boundary $\partial\Omega$. 
At each point 
of the boundary, the gradient of $f$ blows up and hence the graph of $f$ 
has a vertical tangent plane. Han, Shen and Wang \cite{Han2016} proved that 
$f\in C^{\frac{1}{n+1}}(\bar\Omega)$.

Geometrically, it is more interesting to discuss the 
regularity of the graph of $f$ instead of the regularity of $f$ itself. 
Lin \cite{Lin1989Invent} and Tonegawa
\cite{Tonegawa1996MathZ} proved the following result. 
{\it If $\partial\Omega$ is $C^{n,\alpha}$ for some $\alpha\in (0,1)$, 
then the graph of $f$ is $C^{n,\alpha}$ up to the boundary. 
If $\partial\Omega$ is smooth, then the graph of $f$ is smooth up to the boundary 
if the dimension $n$ is even or if the dimension $n$ is odd and 
the principal curvatures of $\partial\Omega$ satisfy a differential equation of order $n+1$.}
See also \cite{Lin2012Invent}. 

Lin \cite{Lin1989Invent} conjectured that the graph of $f$ is analytic up to the boundary 
if $\partial\Omega$ is analytic. In this paper, we prove this conjecture under the necessary extra assumption that 
the graph is smooth up to the boundary. The first result is given by 
the following theorem. 

\begin{theorem}\label{thrm-Analyticity-MinimalGraph} 
Let $\Omega$ be a bounded smooth domain in $\mathbb{R}^n$ with 
$H_{\partial\Omega}\ge 0$ and $f\in C(\bar \Omega)\cap C^\infty(\Omega)$ be a 
solution of \eqref{Eqf}-\eqref{EqfCondition}. 
Assume $\partial\Omega$ is analytic near $x_0\in\partial\Omega$. 
Then, 

$\operatorname{(1)}$ for $n$ even, the graph of $f$ is analytic up to $\partial\Omega$
near $x_0$; 

$\operatorname{(2)}$ for $n$ odd, the graph of $f$ is analytic up to $\partial\Omega$
near $x_0$ if it is smooth up to $\partial\Omega$ near $x_0$.
\end{theorem} 

Locally near each boundary point, the graph of $f$ can be represented by a function 
over its vertical tangent plane. Specifically, 
we fix a boundary point of $\Omega$, say the origin, and assume that 
the vector $e_n=(0,\cdots, 0,1)$ is the interior normal vector to $\partial\Omega$ 
at the origin. Then, with $x=(x',x_n)$, the $x'$-hyperplane is the tangent plane of 
$\partial\Omega$ at the origin, and the boundary $\partial\Omega$ can be expressed 
in a neighborhood of the origin as a graph of a smooth function over $\mathbb R^{n-1}\times\{0\}$, 
say 
$$x_n=\varphi(x').$$
We now denote points in $\mathbb R^{n+1}=
\mathbb R^n\times\mathbb R$ by $(x',x_n,y_n)$. The vertical hyperplane 
given by $x_n=0$ is the tangent plane to the graph of $f$ at the origin in $\mathbb R^{n+1}$,
and we can represent the graph of $f$ as a graph of a new function $u$ defined in terms of 
$(x', 0, y_n)$ for small $x'$ and $y_n$, with $y_n>0$. In other words, we treat 
$\mathbb R^n=\mathbb R^{n-1}\times\{0\}\times\mathbb R$ as our new base space and write 
$u=u(y)=u(y', y_n)$, with $y'=x'$. Then, for some $R>0$,
$u$ satisfies 
\begin{align}\label{eq-Intro-Equ}
\Delta u - \frac{u_i u_j}{1+|D u|^2}u_{ij}-\frac{n u_{n}}{y_n}=0  \quad \text{in } B_R^+,
\end{align}
and 
\begin{align}\label{eq-Intro-EquCondition}
u=\varphi\quad\text{on }B_R'.
\end{align}
Lin \cite{Lin1989Invent} and Tonegawa
\cite{Tonegawa1996MathZ} proved their regularity result for the graph of $f$ 
by proving the corresponding 
regularity of $u$. Concerning the analyticity of $u$ up to the boundary, we have 
the following result. 

\begin{theorem}\label{thrm-Analyticity-VerticalGraph} 
Let $\varphi$ be an analytic function in $B_R'$ 
and $u\in C(\bar B^+_R)\cap C^\infty(B^+_R)$ be 
a solution of \eqref{eq-Intro-Equ}-\eqref{eq-Intro-EquCondition}. Then, for any $r\in (0, R)$, 

$\operatorname{(1)}$ for $n$ even, $u$ is analytic in $\bar B_{r}^+$; 

$\operatorname{(2)}$ for $n$ odd, $u$ is analytic in $y', y_n$ and $y_n\log y_n$ for 
$(y',y_n)\in \bar B_{r}^+$. In addition, if $u$ is smooth in $\bar B_{R}^+$, then 
$u$ is analytic in $\bar B_{r}^+$.\end{theorem} 

Obviously, Theorem \ref{thrm-Analyticity-VerticalGraph} 
implies Theorem \ref{thrm-Analyticity-MinimalGraph} and asserts that the minimal surfaces in the hyperbolic space are 
analytic up to their boundary at infinity if they are smooth up to boundary.

At the first glance, it seems strange that $y_n\log y_n$ appears in the function of $u$. In fact,  
the logarithmic factor shows up in the expansion of $u$ near the boundary $y_n=0$ and is the obstruction 
for $u$ being smooth up to the boundary. 

In \cite{HanJiang2014}, we studied the expansion of $u$ near $y_n=0$. 
Let $k\ge n+1$ be an integer and set, for $n$ even, 
\begin{align}\label{b1a-v}
u_k=\varphi+c_2y_n^2+c_4y_n^4+\cdots
+c_{n}y_n^{n}+\sum_{i=n+1}^k
c_{i} y_n^i, 
\end{align}
and, for $n$ odd, 
\begin{align}\label{b1b-v}
u_k=\varphi+c_2y_n^2+c_4y_n^4+\cdots
+c_{n-1}y_n^{n-1}+\sum_{i=n+1}^k
\sum_{j=0}^{\left[\frac{i-1}{n}\right]}c_{i,j} y_n^i (\log y_n)^j,
\end{align}
where $c_i$ and $c_{i,j}$ are %smooth 
functions of $y'\in B_R'$. In the summation in $u_k$, the lowest order is $y_{n}^{n+1}$ if $n$ 
is even and is $y_n^{n+1}\log y_n$ if $n$ is odd, and 
the highest order is given by $y_n^k$. 
According to the pattern in this expansion, if we intend to continue to 
expand $u_k$, the next term has an order of $y_n^{k+1}$ for $n$ even
and  $y_n^{k+1}(\log y_n)^{\left[\frac{k}{n}\right]}$ for $n$ odd. 
In \cite{HanJiang2014}, we estimated the remainder $u-u_k$ for 
appropriately determined coefficients $c_i$ and $c_{i,j}$. 
The coefficient $c_{n+1}$ in \eqref{b1a-v} or $c_{n+1,0}$ in  \eqref{b1b-v} is the coefficient of the first 
nonlocal term. The coefficients of lower order terms, $c_2, \cdots, c_n$ for $n$ even and 
$c_2, \cdots, c_{n-1}$ and $c_{n+1,1}$ for $n$ odd, 
can be expressed explicitly in terms of $\varphi$ and are refereed to as
{\it local terms}. 

If $\varphi$ is smooth, then $u_k$ is a smooth function in $B_R^+$ if $n$ is even and is not 
necessarily smooth in $y_n$ 
if $n$ is odd due to the presence of $\log y_n$. The first logarithmic term is given by 
$y_n^{n+1}\log y_{n}$ with the coefficient $c_{n+1, 1}$. It is proved in \cite{HanJiang2014}
that there are no logarithmic terms in $u_k$ if $c_{n+1,1}=0$. 
%In fact, 
%$c_{n+1,1}=0$ is the differential equation for principal curvatures of $\partial\Omega$
%as in Theorem A(2). 
For $n=3$, $c_{4,1}=0$ if and only if $\partial\Omega$ is a 
Willmore surface. 

We prove Theorem \ref{thrm-Analyticity-VerticalGraph}, or more general 
Theorem \ref{Thm-MainThm}, in two steps. 
In the first step, we prove that all coefficients $c_i, c_{i,j}$ in  \eqref{b1a-v} and \eqref{b1b-v}
are analytic in $B_R'$. The crucial part is to prove the analyticity of the first nonlocal coefficient, 
$c_{n+1}$ in \eqref{b1a-v} or $c_{n+1,0}$ in  \eqref{b1b-v}. 
In the second step, we prove 
$u_k\to u$ uniformly in $B_{r}^+$ as $k\to\infty$, for any $r\in (0,R)$. 
In this step, we follow techniques in \cite{Kichenassamy2004Adv}, \cite{KL:1}, \cite{KL:2}, and \cite{Nirenberg1972},

Logarithmic terms in the boundary expansions also appear in other problems, such as 
the singular Yamabe problem 
in \cite{ACF1982CMP}, \cite{Loewner&Nirenberg1974} and
\cite{Mazzeo1991}, the complex Monge-Amp\`{e}re equations in \cite{ChengYau1980CPAM}, 
\cite{Fefferman1976} and \cite{LeeMelrose1982}, 
and the asymptotically hyperbolic Einstein metrics
in \cite{Anderson2003}, \cite{Biquad2010}, 
\cite{Chrusciel2005} and 
\cite{Hellimell2008}, and many other problems. 
%In fact, Fefferman \cite{Fefferman1976}
%observed that logarithmic terms should appear 
%in the expansion. 
A result similar to Theorem \ref{thrm-Analyticity-VerticalGraph}  
holds for the singular Yamabe problem, which is also a consequence of Theorem \ref{Thm-MainThm}. 

Kichenassamy \cite{Kichenassamy2004Adv} constructed solutions in the form of convergent series to the 
 local embedding of an arbitrary real analytic
manifold, of even dimension $n$, into a Ricci-flat manifold of dimension $n+2$ admitting
a homothety. See also \cite{KL:1}, \cite{KL:2}. For the series \eqref{b1b-v}, his result can be reformulated as the following: 
given analytic functions $\varphi$ and $c_{n+1, 0}$ on $B_R'$, the series in \eqref{b1b-v} converges 
to a solution of \eqref{eq-Intro-Equ}-\eqref{eq-Intro-EquCondition} in 
$B_r^+$, for any $r\in (0,R)$. One of the main contributions in this paper is the analyticity of $c_{n+1, 0}$ 
for any solution $u$ in $B_R^+$ with an analytic boundary value $\varphi$ on $B_R'$. 

We finish the introduction with a brief outline of the paper. 
In Section \ref{sec-TangentialAnalyticity}, we prove that $u$ is tangentially analytic 
by the maximum principle and a scaling argument. In Section \ref{sec-AnalyticCoefficients},  
we prove that all coefficients in \eqref{b1a-v} and \eqref{b1b-v} are analytic
by studying the expressions of 
those coefficients. 
In Section \ref{Sec-NormalAnalyticity}, we prove that $u$ is analytic in $y_n$ and $y_n\log y_n$. 
The proof is based on an iteration of solutions of the corresponding ordinary differential equations. 
We carried out the proof for a class of quasilinear elliptic equations more general than 
the equation \eqref{eq-Intro-Equ}. In Section \ref{Sec-LN}, we discuss the Loewner-Nirenberg problem 
briefly. 

We would like to thank Robin Graham and 
Satyanad Kichenassamy for helpful discussions.

\section{The Tangential Analyticity}\label{sec-TangentialAnalyticity}

In this section, we discuss the analyticity along the tangential directions. 

We denote by $x=(x', t)$ points in $\mathbb R^n$, with $x_n=t$, and, set, for any constant $r>0$, 
$$G_r=\{(x',t):\, |x'|<r,\, 0<t<r\}.$$
We assume $u\in C(\bar G_1)\cap C^\infty(G_1)$ satisfies  
\begin{align}\label{eq-QuasiMain}
A_{ij} u_{ij}+P \frac{u_t}{t}+Q \frac{u}{t^2}+N=0\quad\text{in }G_1, 
\end{align}
where $A_{ij}, P, Q$ and $N$ are functions of the form 
\begin{align*} 
A_{ij}=A_{ij}\left(x',t,Du, \frac{u}{t}\right),\quad P=P\left(x',t,Du, \frac{u}{t}\right),\quad
Q=Q\left(x',t,Du, \frac{u}{t}\right),\end{align*}
and 
\begin{align*}N=N\left(x', t, Du, \frac{u}{t}\right).\end{align*}
%\begin{align*}N=N\left(x', t, Du, \frac{u}{t}, \frac{|D_{x'}u|^2}{t}\right).\end{align*}
%Hence,  $A_{ij}, P, Q$ and $N$ are functions of $x',$ 
%$t$ and 
%\be\label{eq-KeyExpressions}Du,\, \frac{u}{t},\, \frac{|D_{x'}u|^2}{t}.\ee
We assume \eqref{eq-QuasiMain} is uniformly elliptic; namely, there exists a positive constant 
$\lambda$ such that, for any $(x',t,p,s)\in G_1\times \mathbb R^n\times\mathbb R$ and 
any $\xi\in\mathbb R^n$, 
$$\lambda^{-1}|\xi|^2\le A_{ij}(x',t,p,s)\xi_i\xi_j\le \lambda|\xi|^2.$$
Concerning the solution $u$, we always assume, for some positive constant $C_0$,  
\be\label{eq-QuasiCondition_1}
\left|u\right|\le C_0t^2,\ee
and 
\be\label{eq-QuasiCondition_2}|Du|\le C_0t.\ee

Throughout this section, we always assume that 
$A_{ij}, P, Q$ and $N$ are analytic in $\bar G_1\times\mathbb R^n\times \mathbb R$. 
For convenience, we denote by 
$(x,y)\in G_1\times\mathbb R^{n+1}$ 
the arguments of $A_{ij}, P, Q$ and $N$.  We assume, for any nonnegative integers $k$ and $l$,  
\begin{equation}\label{eq-QuasiLinearAnalyticity-Assum1}
|D^{k+l}_{(x,y)}A_{ij}|+|D^{k+l}_{(x,y)}P|+|D^{k+l}_{(x,y)}Q|\le A_0A^{k+l}(k-2)!(l-2)!
\quad\text{in }G_1\times\mathbb R^{n+1},\end{equation}
and
\begin{equation}\label{eq-QuasiLinearAnalyticity-Assum2}
|D^{k+l}_{(x,y)}N|\le A_0A^{k+l}(k-2)!(l-2)!\quad\text{in }G_1\times\mathbb R^{n+1},\end{equation}
for some positive constants $A_0$ and $A$. Here and hereafter, $m!=1$ for any integer $m\le 0$. 

\begin{theorem}\label{thrm-TangentialAnalyticity}
Let $A_{ij}, P, Q$ and $N$ be smooth in its arguments and satisfy 
\eqref{eq-QuasiLinearAnalyticity-Assum1} and \eqref{eq-QuasiLinearAnalyticity-Assum2}, 
and $u\in C^1(\bar G_1)\cap C^\infty(G_1)$ be a solution of \eqref{eq-QuasiMain} 
and satisfy \eqref{eq-QuasiCondition_1}
and \eqref{eq-QuasiCondition_2}. 
Assume, for some $c_0>0$,  
\be\label{eq-NegativeCondition}2 A_{nn}+2P+Q\le -c_0\quad\text{in } G_1\times\mathbb R^{n+1}.\ee
Then, 
there exists a positive constant $R\in (0,1)$, depending only on $n$, $A_0$,  $A$ 
and $c_0$, such that, 
for any $(x',t)\in G_R$ and any nonnegative integer $l$, 
\begin{align}
\label{eq-LinearAnalyticity-Tangential1} |D^l_{x'} u(x',t)|&\leq B_0B^{(l-1)^+} (l-1)!  t^2(R-|x'|)^{-(l-1)^+},\\
\label{eq-LinearAnalyticity-Tangential2} |D D^l_{x'} u(x',t)|&\leq B_0 B^{(l-1)^+}(l-1)!t (R-|x'|)^{-(l-1)^+},\\
\label{eq-LinearAnalyticity-Tangential3} |D^2 D^l_{x'} u(x',t)|&\leq B_0B^{(l-1)^+}(l-1)! (R-|x'|)^{-(l-1)^+}, 
\end{align}
where $B_0$ and $B$ are positive constants depending only on $n$, $A_0$, $A$ and $c_0$. 
\end{theorem} 

\begin{proof} We note that the equation \eqref{eq-QuasiMain}  is  elliptic 
wherever $t$ is positive. Hence by the interior analyticity, we assume that, for any 
$R\in (0,1)$, \eqref{eq-LinearAnalyticity-Tangential1}-\eqref{eq-LinearAnalyticity-Tangential3}
hold in $B_R'\cap \{R/2\le t\le R\}$, for some constants $B_0$ and $B$. 

For each positive integer $l$, define 
$$T_l=\left\{(x',t)\in G_R: t<\frac{1}{l}(R-|x'|)\right\}.$$ 
Hence, $T_l$ is a circular cone and shrinks while $l$ increases. 
In this way, we decompose $G_R$ into two parts $T_l$ and $G_R\setminus T_l$.

In the following, we set 
$$L=A_{ij} \partial_{ij}+\frac{P}{t}\partial_t+\frac{Q}{t^2}.$$
By applying $D_{x'}^l$ to \eqref{eq-QuasiMain}, we obtain 
\begin{align}\label{TS91}
L(D_{x'}^lu)+N_l=0,
\end{align}
where $N_l$ is given by 
\be\label{eq-KeyExpressions2}N_l=
\sum_{m=0}^{l-1}
\left(\begin{matrix}l\\m\end{matrix}\right)\left(D_{x'}^{l-m}A_{ij}\cdot D_{x'}^mu_{ij}+
D_{x'}^{l-m}P\cdot\frac{D_{x'}^mu_t}{t}+D_{x'}^{l-m}Q\cdot\frac{D_{x'}^mu}{t^2}\right)
+D_{x'}^lN.\ee
Derivatives of $A_{ij}, P, Q$ and $N$ also result in derivatives of 
$u$. 
%For example, the $l$-th derivative of  the last argument in $N$ yields 
%$$D_{x'}^l\left(\frac{|D_{x'}u|^2}{t}\right)=\sum_{k+m= l}\left(\begin{matrix}l\\m\end{matrix}\right)
%\frac{D_{x'}^{k+1}uD_{x'}^{m+1}u}{t}.$$

We now prove \eqref{eq-LinearAnalyticity-Tangential1}-\eqref{eq-LinearAnalyticity-Tangential3}
by induction. 
By \eqref{eq-NegativeCondition}, we can apply 
Theorem 4.2 in \cite{HanJiang2014} and obtain 
\eqref{eq-LinearAnalyticity-Tangential1}-\eqref{eq-LinearAnalyticity-Tangential3}
for $l=0, 1$ and $R=1$. 
Let $p\ge 2$ be an integer and
assume \eqref{eq-LinearAnalyticity-Tangential1}-\eqref{eq-LinearAnalyticity-Tangential3} hold for all $l< p$.

{\it Step 1.} We prove \eqref{eq-LinearAnalyticity-Tangential1} for $l=p$ in $G_R$. We consider 
the cases $T_p$ and $G_R\setminus T_p$ separately. 

We first take an $x_0=(x_0', t_0)\in T_p$. Set 
$$\rho=\frac1p(R-|x_0'|),$$ 
and 
$$G_\rho(\tilde x_0)=\{(x',t):\, |x'-x_0'|<\rho,\, t\in (0,\rho)\},$$
where $\tilde x_0=(x_0',0)$. The definition of $T_p$ implies $t_0<\rho$. 
Next, we take any $(x',t)\in G_\rho(\tilde x_0)$. Then,  
$$(R-|x_0'|)-(R-|x'|)=|x'|-|x_0'|\le |x'-x_0'|<\rho=\frac1p(R-|x'_0|).$$
Hence, 
\begin{equation}\label{eq-LinearAnalyticity-Relation}R-|x_0'|<\frac{p}{p-1}(R-|x'|).\end{equation}
A similar argument yields 
\begin{equation}\label{eq-LinearAnalyticity-Relation2}R-|x'|<\frac{p+1}{p}(R-|x_0'|).\end{equation}
With $t<\rho$ in $G_\rho(\tilde x_0)$, we have 
$$t<\rho=\frac{1}{p}(R-|x_0'|)<\frac{1}{p-1}(R-|x'|).$$
This implies 
\begin{equation}\label{eq-LinearAnalyticity-Inclusion}G_\rho(\tilde x_0)\subset T_{p-1}.\end{equation}

Consider, for some positive constant $\varepsilon$  to be determined, 
$${w}(x',t)=M(\varepsilon|x-x_0'|^2+t^2).$$
Then, 
\begin{align}
L w=M\left(2 A_{nn}+2P+Q + 2\varepsilon\sum_{\alpha=1}^{n-1}A_{\alpha\alpha}\right).
\label{TS72}
\end{align}
By \eqref{eq-NegativeCondition} and taking $\varepsilon$ small, we have 
\begin{align*}Lw\le -\frac12Mc_0.\end{align*}
For simplicity, we assume $c_0\in (0,1]$. 
Next, the definition of $w$ implies  
\begin{align*} 
w&\ge M\varepsilon \rho^2\quad\text{on }\partial B_\rho'(x_0')\times (0,\rho),\\
w&\ge M\rho^2\quad\text{on }B_\rho'(x_0')\times\{\rho\}.\end{align*}
By the induction hypotheses 
\eqref{eq-LinearAnalyticity-Tangential2} for $l=p-1$, we have 
$$|D_{x'}^pu(x)|\le B_0B^{p-2}(p-2)!t(R-|x'|)^{-p+2}.$$
Note that \eqref{eq-LinearAnalyticity-Relation} implies, for $(x',t)\in G_\rho(\tilde x_0)$,  
$$(R-|x'|)^{-p+2}<\left(\frac{p-1}{p}\right)^{-p+2}(R-|x_0'|)^{-p+2}
\le c_1(R-|x_0'|)^{-p+2},$$
where $c_1$ is a positive constant independent of $p$. Hence by the definition of $\rho$, we get, 
for any $(x',t)\in G_\rho(\tilde x_0)$, 
$$\aligned |D_{x'}^pu(x)|&\le c_1B_0B^{p-2}(p-2)!\rho(R-|x_0'|)^{-p+2}\\
&= c_1B_0B^{p-2}(p-1)!\rho^2(R-|x_0'|)^{-p+1}.\endaligned$$
In order to have $w\ge |D_{x'}^pu|$ on $\partial G_\rho(\tilde x_0)$, we need to choose, 
by renaming $c_1$,  
\begin{equation}\label{eq-LinearAnalyticity-ChoiceM1}
M\ge c_1B_0B^{p-2}(p-1)!(R-|x_0'|)^{-p+1}.\end{equation}
Next, we consider \eqref{TS91}
and estimate \eqref{eq-KeyExpressions2} for $l=p$. 
We claim that, by taking $B$ sufficiently large depending only on 
$A_0$, $B_0$ and $A$, we have, for any $(x',t)\in G_\rho(\tilde x_0)$,
\begin{equation}\label{eq-LinearAnalyticity-Estimate-f}
|N_p(x',t)|
\le C_1 B_0B^{p-2}(p-1)!(R-|x_0'|)^{-p+1},
\end{equation}
where $C_1$ is a positive constant depending only on $A_0$, $B_0$ and $A$. By renaming 
$C_1$, we may require $C_1\ge c_1$, for $c_1$ in \eqref{eq-LinearAnalyticity-ChoiceM1}, 
and $C_1\ge 2c_0^{-1}$, for $c_0$ as in \eqref{eq-NegativeCondition}. 
Set 
\begin{equation}\label{eq-LinearAnalyticity-ChoiceM2}
M= C_1 B_0B^{p-2}(p-1)!(R-|x_0'|)^{-p+1}.
\end{equation}
Therefore, we obtain 
$$\aligned 
L(\pm D_{x'}^pu)&\ge Lw\quad\text{in }G_\rho(\tilde x_0),\\
\pm D_{x'}^pu&\le w\quad\text{on }\partial G_\rho(\tilde x_0).\endaligned$$
By the maximum principle, we have 
$$\pm D_{x'}^pu\le w\quad\text{in }G_\rho(\tilde x_0),$$ 
or 
$$|D_{x'}^pu|\le w\quad\text{in }G_\rho(\tilde x_0).$$
By taking $x'=x'_0$, we obtain, for any $(x_0',t)\in G_\rho(\tilde x_0)$, 
$$|D_{x'}^pu(x_0',t)|\le Mt^2.$$
In conclusion, by \eqref{eq-LinearAnalyticity-ChoiceM2},
we obtain, 
for any $(x',t)\in T_p$, 
$$|D^p_{x'} u(x',t)|\leq C_1 B_0B^{p-2}(p-1)!t^2(R-|x'|)^{-p+1}.$$

We now prove \eqref{eq-LinearAnalyticity-Estimate-f}. In view of \eqref{eq-KeyExpressions2} 
with $l=p$, we first estimate $D_{x'}^pN$. 
For any $k=1, \cdots, p$, by taking $l=k-1<p$ in the induction hypothesis 
\eqref{eq-LinearAnalyticity-Tangential2} and \eqref{eq-LinearAnalyticity-Tangential3}, we have 
$$\aligned 
\left|D_{x'}^k\frac{u}{t}\right|&\le\frac1t\left|D_{x'}^{k-1}Du\right|
\le B_0B^{(k-2)^+}(k-2)!(R-|x'|)^{-(k-2)^+},\\
|D_{x'}^kDu|&\le|D_{x'}^{k-1}D^2u|\le B_0B^{(k-2)^+}(k-2)!(R-|x'|)^{-(k-2)^+}.
\endaligned$$
By Lemma \ref{lemma-Composition} and Remark \ref{rmk-Composition}, we obtain 
\begin{equation}\label{eq-EstimateN1}|D^p_{x'}N|\le \widetilde B_0B^{p-2}(p-2)!(R-|x'|)^{-(p-2)}.
\end{equation}
Next, we estimate terms involving $A_{ij}$ in \eqref{eq-KeyExpressions2}, i.e., 
\begin{equation}\label{eq-estimate-I}I=\sum_{m=0}^{p-1}
\left(\begin{matrix}p\\m\end{matrix}\right)D_{x'}^{p-m}A_{ij} \partial_{ij}D_{x'}^mu.\end{equation}
Similar as \eqref{eq-EstimateN1}, we have, for any $k=0, 1, \cdots, p$, 
$$|D^k_{x'}A_{ij}|\le \widetilde B_0B^{(k-2)^+}(k-2)!(R-|x'|)^{-(k-2)^+}.$$
In expanding the summation in $I$, we consider $m=0, 1, p-1$ separately. 
By  the induction hypotheses 
\eqref{eq-LinearAnalyticity-Tangential3} for $l<p$, we have 
$$\aligned 
|I|&\le \widetilde B_0B^{p-2}(p-2)!(R-|x'|)^{-p+2}+\widetilde B_0B_0B^{p-3}p(p-3)!(R-|x'|)^{-p+3}\\
&\qquad +\widetilde B_0B_0B^{p-3}(p-1)!(R-|x'|)^{-p+3}
\sum_{m=2}^{p-2}\frac{p}{m(p-m)(p-m-1)}\\
&\qquad+\widetilde B_0B_0B^{p-2}p(p-2)!(R-|x'|)^{-p+2}.\endaligned$$ 
We note that the last term in the right-hand side above has the order $B^{p-2}(p-1)!$. 
A straightforward calculation yields 
$$|I|\le B_1B_0B^{p-2}(p-1)!(R-|x'|)^{-p+1}.$$
We have similar results for other terms in $N_p$ by employing 
\eqref{eq-LinearAnalyticity-Tangential1} and \eqref{eq-LinearAnalyticity-Tangential2} for $l<p$. 
Therefore, we obtain \eqref{eq-LinearAnalyticity-Estimate-f}. 

Next, we take $(x',t)\in G_R\setminus T_p$. By the induction hypotheses 
\eqref{eq-LinearAnalyticity-Tangential2} for $l=p-1$, we have 
$$|D^p_{x'} u(x',t)|\leq B_0B^{p-2}(p-2)!t (R-|x'|)^{-p+2}.$$
Note $R-|x'|\le pt$ in $G_R\setminus T_p$. Then, 
$$|D^p_{x'} u(x',t)|\leq \frac{p}{p-1}B_0B^{p-2}(p-1)!t^2 (R-|x'|)^{-p+1}.$$

By combining the both cases for points in $T_p$ and $G_R\setminus T_p$, 
we obtain, for any $(x', t)\in G_R$, 
\begin{equation}\label{eq-LinearAnalyticity-Tangential1-tau}
|D^p_{x'} u(x',t)|\leq C_1 B_0B^{p-2}(p-1)! t^2(R-|x'|)^{-p+1}.\end{equation}
This implies \eqref{eq-LinearAnalyticity-Tangential1} for $l=p$, if $B\ge C_1$. 
The extra factor $B^{-1}$ is for later purposes.

\smallskip

{\it Step 2.} We prove \eqref{eq-LinearAnalyticity-Tangential2} for $l=p$ in $G_R$. Again, we consider 
the cases $T_p$ and $G_R\setminus T_p$ separately.

Take any $x_0=(x_0', t_0)\in T_p$ and set $\rho=t_0$. Then, $B_\rho(x_0)\subset G_R$. 
By a similar argument, \eqref{eq-LinearAnalyticity-Relation}
and \eqref{eq-LinearAnalyticity-Relation2} hold in $B_\rho(x_0)$. Similar to 
\eqref{eq-LinearAnalyticity-Estimate-f}, we have, in $B_\rho(x_0)$,  
\begin{equation}\label{eq-LinearAnalyticity-Estimate-f2}
|N_p|
\le C_1 B_0B^{p-2}(p-1)!(R-|x_0'|)^{-p+1}.
\end{equation}
We now consider \eqref{TS91} in $B_{3\rho/4}(x_0)$ for $l=p$. 
Note 
$$|A_{ij}|_{L^\infty(B_{3\rho/4}(x_0))}
+\rho\left|t^{-1}P\right|_{L^\infty(B_{3\rho/4}(x_0))}
+\rho^2\left|t^{-2}Q\right|_{L^\infty(B_{3\rho/4}(x_0))}\le C.$$
We  fix an arbitrary constant $\alpha\in(0,1)$. 
The scaled $C^{1,\alpha}$-estimate implies 
$$\aligned &\rho^{\alpha}[D_{x'}^pu]_{C^\alpha(B_{\rho/2}(x_0))}
+\rho|DD_{x'}^pu|_{L^\infty(B_{\rho/2}(x_0))}
+\rho^{1+\alpha}[DD_{x'}^pu]_{C^\alpha(B_{\rho/2}(x_0))}\\
&\qquad \le c_2\left(|D_{x'}^pu|_{L^\infty(B_{3\rho/4}(x_0))}+\rho^2|N_p|_{L^\infty(B_{3\rho/4}(x_0))}\right).
\endaligned$$
By \eqref{eq-LinearAnalyticity-Tangential1-tau} and 
\eqref{eq-LinearAnalyticity-Estimate-f2}, we have 
\begin{align}\label{eq-LinearAnalyticity-Estimate-u}\begin{split}
&\rho^{\alpha}[D_{x'}^pu]_{C^\alpha(B_{\rho/2}(x_0))}
+\rho|DD_{x'}^pu|_{L^\infty(B_{\rho/2}(x_0))}
+\rho^{1+\alpha}[DD_{x'}^pu]_{C^\alpha(B_{\rho/2}(x_0))}\\
&\qquad \le C_2 B_0B^{p-2}(p-1)!\rho^2(R-|x_0'|)^{-p+1}.
\end{split}\end{align}
In particular, we get 
$$|DD_{x'}^pu(x_0)|\le C_2  B_0B^{p-2}(p-1)!\rho(R-|x_0'|)^{-p+1}.$$

Next, we take $(x',t)\in G_R\setminus T_p$. By the induction hypotheses 
\eqref{eq-LinearAnalyticity-Tangential3} for $l=p-1$, we have 
$$|DD^p_{x'} u(x',t)|\leq B_0B^{p-2}(p-2)! (R-|x'|)^{-p+2}.$$
Note $R-|x'|\le pt$ in $G_R\setminus T_p$. Then, 
$$|DD^p_{x'} u(x',t)|\leq \frac{p}{p-1}B_0B^{p-2}(p-1)!t (R-|x'|)^{-p+1}.$$

By combining the both cases for points in $T_p$ and $G_R\setminus T_p$, 
we obtain, for any $(x', t)\in G_R$, 
\begin{equation}\label{eq-LinearAnalyticity-Tangential2-tau}
|DD^p_{x'} u(x',t)|\leq C_2 B_0B^{p-2}(p-1)! t(R-|x'|)^{-p+1}.\end{equation}
This implies \eqref{eq-LinearAnalyticity-Tangential2} for $l=p$, if $B\ge C_2$. 

{\it Step 3.} We prove \eqref{eq-LinearAnalyticity-Tangential3} in $T_p$ for $l=p$. 

%To prove \eqref{eq-LinearAnalyticity-Tangential3} for $l=p$ in $T_p$, we need 
%to employ the $C^{2,\alpha}$ estimates. 
As in Step 2, we take any $x_0=(x_0',t_0)\in T_p$ and set $\rho=t_0$. 
A simple calculation yields
$$\rho^\alpha[A_{ij}]_{C^\alpha(B_{\rho/2}(x_0))}
+\rho^{1+\alpha}[t^{-1}P]_{C^\alpha(B_{\rho/2}(x_0))}
+\rho^{2+\alpha}[t^{-2}Q]_{C^\alpha(B_{\rho/2}(x_0))}\le c_3.$$
We now consider \eqref{TS91} in $B_{\rho/2}(x_0)$ for $l=p$. 
The scaled $C^{2,\alpha}$-estimate implies 
$$\rho^2|D^2D_{x'}^p(x_0)|\le 
c_3\left\{|D_{x'}^pu|_{L^\infty(B_{\rho/2}(x_0))}+\rho^2|N_p|_{L^\infty(B_{\rho/2}(x_0))}
+\rho^{2+\alpha}[N_p]_{C^\alpha(B_{\rho/2}(x_0))}\right\}.$$
By \eqref{eq-LinearAnalyticity-Tangential1-tau} and 
\eqref{eq-LinearAnalyticity-Estimate-f2}, we have 
$$|D^2D_{x'}^p(x_0)|\le 
C_3 B_0B^{p-2}(p-1)!(R-|x_0'|)^{-p+1}
+c_3\rho^{\alpha}[N_p]_{C^\alpha(B_{\rho/2}(x_0))}.$$
We claim 
\begin{equation}\label{eq-LinearAnalyticity-Estimate-f3}
\rho^{\alpha}[N_p]_{C^\alpha(B_{\rho/2}(x_0))}\le C_3 B_0B^{p-2}(p-1)!(R-|x_0'|)^{-p+1}.\end{equation}
Hence, 
$$|D^2D_{x'}^p(x_0)|\le 
C_3 B_0B^{p-2}(p-1)!(R-|x_0'|)^{-p+1}.$$
By taking $B\ge C_3$, we obtain, for any $(x',t)\in T_p$, 
$$|D^2D_{x'}^pu(x',t)|\le B_0B^{p-1}(p-1)!(R-|x_0'|)^{-p+1}.$$ 
This is \eqref{eq-LinearAnalyticity-Tangential3} for $l=p$ in $T_p$. 

We now  prove \eqref{eq-LinearAnalyticity-Estimate-f3}
by examining $N_p$ given by \eqref{eq-KeyExpressions2} for $l=p$. We note 
that $N_p$ consists of two parts. The first part is given by a summation 
and the second part 
by $D_{x'}^pN$. For $D_{x'}^pN$, we have 
\begin{equation}\label{eq-LinearAnalyticity-Estimate-N1}
\rho^{\alpha}[D_{x'}^pN]_{C^\alpha(B_{\rho/2}(x_0))}
\le \widetilde B_0B^{p-2}(p-2)!(R-|x_0'|)^{-(p-2)}.\end{equation}
The proof is similar to that of \eqref{eq-EstimateN1}. We point out that 
Lemma \ref{lemma-Composition} still holds if the $L^\infty$-norms are 
replaced by $C^\alpha$-norms and the needed estimates of the 
$C^\alpha$ semi-norms of $DD_{x'}^lu$ and 
$D_{x'}^lu/t$ are provided by \eqref{eq-LinearAnalyticity-Estimate-u}, for $l\le p$. 
Next, we examine the summation part in $N_p$ and discuss $I$ in 
\eqref{eq-estimate-I} for an illustration. Similar to 
\eqref{eq-LinearAnalyticity-Estimate-N1}, we have, for any $l\le p$,  
$$\rho^{\alpha}[D_{x'}^lA_{ij}]_{C^\alpha(B_{\rho/2}(x_0))}
\le \widetilde B_0B^{(l-2)^+}(l-2)!(R-|x_0'|)^{-(l-2)^+}.$$
We note 
that $I$ is a linear combination of $D^2D_{x'}^{m}u$, for $m\le p-1$, 
which can be written as $DD_{x'}^mu$ for $m\le p$ and 
$\partial_t^2D_{x'}^mu$ for $m\le p-1$. We estimate these two groups separately. 
To do this, we first have, for any $l\le p$, 
\begin{align}\label{eq-LinearAnalyticity-Estimate-u0}\begin{split}
&|D_{x'}^lu|_{L^\infty(B_{\rho/2}(x_0))}
+\rho^{\alpha}[D_{x'}^lu]_{C^\alpha(B_{\rho/2}(x_0))}\\
&\qquad+\rho|DD_{x'}^lu|_{L^\infty(B_{\rho/2}(x_0))}
+\rho^{1+\alpha}[DD_{x'}^lu]_{C^\alpha(B_{\rho/2}(x_0))}\\
&\qquad \le C_2 B_0B^{(l-2)^+}(l-1)!\rho^2(R-|x_0'|)^{-(l-1)^+}.
\end{split}\end{align}
We note that \eqref{eq-LinearAnalyticity-Estimate-u0} is implied by 
\eqref{eq-LinearAnalyticity-Tangential1-tau} and 
\eqref{eq-LinearAnalyticity-Estimate-u} for $l=p$. The proof in Step 2 actually shows 
that \eqref{eq-LinearAnalyticity-Estimate-u0} holds for all $l\le p$. 
Next, we prove, for $l\le p-1$, 
\begin{align}\label{eq-LinearAnalyticity-Estimate-u0a}\begin{split}
&\rho|\partial_t^2D_{x'}^lu|_{L^\infty(B_{\rho/2}(x_0))}
+\rho^{1+\alpha}[\partial_t^2D_{x'}^lu]_{C^\alpha(B_{\rho/2}(x_0))}\\
&\qquad \le C_3B_0B^{(l-1)^+}l!\rho^2(R-|x_0'|)^{-l}.
\end{split}\end{align}
To prove \eqref{eq-LinearAnalyticity-Estimate-u0a}, we first have, by 
\eqref{eq-QuasiMain},
\begin{equation*}\partial_t^2u=-\frac{N}{A_{nn}}
-\sum_{0\le i+j\le 2n-1}\frac{A_{ij}}{A_{nn}}\partial_{ij}u-\frac1t\frac{P}{A_{nn}}\partial_tu
-\frac1{t^2}\frac{Q}{A_{nn}}u.\end{equation*}
Then, for $l\le p-1$, 
\begin{equation}\label{eq-equation-l}\partial_t^2D_{x'}^{l}u=-D_{x'}^{l}\bigg(\frac{N}{A_{nn}}
+\sum_{0\le i+j\le 2n-1}\frac{A_{ij}}{A_{nn}}\partial_{ij}u+\frac1t\frac{P}{A_{nn}}\partial_tu
+\frac1{t^2}\frac{Q}{A_{nn}}u\bigg).\end{equation}
We analyze the summation involving $A_{ij}$. 
For each pair $i$ and $j$ with $i+j<2n$, $\partial_{ij}u$ is a part of $DD_{x'}u$. 
Hence, for $l\le p-1$, $D_{x'}^l(A_{nn}^{-1}A_{ij}\partial_{ij}u)$ is a linear combination of 
$DD_{x'}^mu$, for $m=1, \cdots, p$. The $C^\alpha$-norms of these derivatives of $u$ 
are already estimated by 
\eqref{eq-LinearAnalyticity-Estimate-u0}. We can analyze other terms similarly. Hence, we have 
\eqref{eq-LinearAnalyticity-Estimate-u0a}. As a consequence, we get 
\begin{equation*}
\rho^{\alpha}[I]_{C^\alpha(B_{\rho/2}(x_0))}\le C_3 B_0B^{p-2}(p-1)!(R-|x_0'|)^{-p+1}.\end{equation*}
We can analyze other terms in $N_p$ similarly. Therefore, we obtain
\eqref{eq-LinearAnalyticity-Estimate-f3} and finish the proof of the claim.

\smallskip

{\it Step 4.} We prove 
\eqref{eq-LinearAnalyticity-Tangential3} in $G_R\setminus T_p$ for $l=p$. We will fix $R$ in this step. 

Take any $x_0=(x_0', t_0)\in G_R\setminus T_p$, with $t_0\le R/2$. Then, $t_0\ge (R-|x_0'|)/p$. 
Set 
$$\rho=\frac{1}{2p}(R-|x_0'|).$$
Then, $t_0\ge 2\rho$. Hence, for any $(x',t)\in B_\rho(x_0)$, $t\ge t_0-\rho\ge \rho$. 
%We have 
%$$\sup_{B_\rho(x_0)}\left|\frac{\rho}{t}\right|\le 1.$$
We now consider \eqref{TS91} in $B_{\rho}(x_0)$ for $l=p+1$. 
Note 
$$|a^{ij}|_{L^\infty(B_{\rho}(x_0))}
+\rho\left|t^{-1}b^i\right|_{L^\infty(B_{\rho}(x_0))}
+\rho^2\left|t^{-2}c\right|_{L^\infty(B_{\rho}(x_0))}\le c_4.$$
We  fix an arbitrary constant $\alpha\in(0,1)$. 
The scaled $C^{1,\alpha}$-estimate implies 
$$\rho|DD_{x'}^{p+1}u(x_0)|\le c_4\left\{|D_{x'}^{p+1}u|_{L^\infty(B_\rho(x_0))}
+\rho^2|N_{p+1}|_{L^\infty(B_\rho(x_0))}\right\}.$$
By the induction hypotheses \eqref{eq-LinearAnalyticity-Tangential3} for $l=p-1$, we have 
$$|D_{x'}^{p+1}u(x)|\le B_0B^{p-2}(p-2)!(R-|x'|)^{-p+2}.$$
By a similar argument, \eqref{eq-LinearAnalyticity-Relation}
and \eqref{eq-LinearAnalyticity-Relation2} hold in $B_\rho(x_0)$. 
Hence, for any $x=(x',t)\in B_\rho(x_0)$, 
\begin{align*}|D_{x'}^{p+1}u(x)|&\le c_4B_0B^{p-2}(p-2)!(R-|x'_0|)^{-p+2}\\
&\le c_4B_0B^{p-1}(p-1)!\rho(R-|x'_0|)^{-p+1}.\end{align*}
Next, we consider \eqref{eq-KeyExpressions2} for $l=p+1$. 
In the expression of $N_{p+1}$, we single out the term $D^2D_{x'}^{p}u$. We note that 
$D_{x'}^lu$, $DD_{x'}^lu$, $D^2D_{x'}^lu$ can be estimated by the induction hypothesis, for $l<p$, 
and that $D_{x'}^pu$, $DD_{x'}^pu$ can be estimated by Step 1 and Step 2, respectively. 
Hence, a similar argument as in Step 1 yields 
$$|N_{p+1}|_{L^\infty(B_\rho(x_0))}\le (p+1)A_0A|D^2D_{x'}^{p}u|_{L^\infty(B_\rho(x_0))}
+C_1 B_0B^{p-2}(p-1)!(R-|x'_0|)^{-p+1}.$$
By a simple substitution, we have 
$$|DD_{x'}^{p+1}u(x_0)|\le (p+1)A_0A\rho|D^2D_{x'}^{p}u|_{L^\infty(B_\rho(x_0))}
+C_1 B_0B^{p-2}(p-1)!(R-|x'_0|)^{-p+1}.$$
Combining with \eqref{eq-equation-l} for $l=p$, we get 
$$|D^2D_{x'}^{p}u(x_0)|\le (p+1)A_0A\rho|D^2D_{x'}^{p}u|_{L^\infty(B_\rho(x_0))}
+C_4 B_0B^{p-2}(p-1)!(R-|x'_0|)^{-p+1}.$$
We now fix a constant $\varepsilon\in (0,1)$. 
By the definition of $\rho$, we can choose $R$ sufficiently small such that
$$|D^2D_{x'}^{p}u(x_0)|\le \varepsilon|D^2D_{x'}^{p}u|_{L^\infty(B_\rho(x_0))}
+C_4 B_0B^{p-2}(p-1)!(R-|x'_0|)^{-p+1}.$$
Next, for any $r\in (0,R)$, we define 
$$h(r)=\sup\{|D^2D_{x'}^{p}u|:\, x\in G_R\setminus T_p,\,  |x'|\le r\}.$$
At points in $B_{\rho}(x_0)\cap T_p$, $D^2 D^{p}_{x'}u$ is already bounded in Step 3. 
Hence, we have, for any $r\in (0,R)$,  
\begin{align*}
h(r) \leq \varepsilon h\big(r+p^{-1}(R-r)\big)+C_4 B_0B^{p-2}(p-1)!(R-r)^{-p+1}.
\end{align*}
By applying Lemma \ref{lemma-LinearAnalyticity-Iteration} below to the function $h$, we obtain, 
for any $r\in (0,R)$, 
$$h(r)\le CC_4 B_0B^{p-2}(p-1)!(R-r)^{-p+1}.$$
We now choose $B\ge CC_4$. For each $(x',t)\in G_R\setminus T_p$, we take 
$r=|x'|$ and then obtain 
$$|D^2D_{x'}^{p}u(x',t)|\le B_0B^{p-1}(p-1)!(R-|x'|)^{-p+1}.$$
This ends the proof of \eqref{eq-LinearAnalyticity-Tangential3} in $G_R\setminus T_p$ for $l=p$. 

In summary, we take $B\ge \max\{C_1, C_2, C_3, CC_4\}$. 
\end{proof}

We need the following lemma to finish the proof of Theorem \ref{thrm-TangentialAnalyticity}.
See Lemma 2 \cite{Friedman1958}. 

\begin{lemma}\label{lemma-LinearAnalyticity-Iteration}
Let $p$ be a positive integer, 
$\varepsilon\in (0,1)$ and $M>0$ be constants, 
and $h(t)$ be a positive monotone increasing function defined in the interval $[0,R]$. 
Assume, for any $r\in (0,R)$, 
\begin{align*}
h(r) \leq \varepsilon h\big(r+p^{-1}(R-r)\big)+M(R-r)^{-p}.
\end{align*}
Then, for any $r\in (0,R)$, 
$$h(r)\le CM(R-r)^{-p},$$ 
where $C$ is a positive constant depending only on $\varepsilon$, independent of $p$.
\end{lemma}

The proof is by a simple iteration and hence is omitted. 

For convenience, we introduce the notion of the tangential analyticity. Let $v$ be a smooth function in 
$\bar G_r$ for some $r>0$. Then, $v$ is {\it tangentially analytic} in $\bar G_r$ if, for any nonnegative 
integer $l$ and any $(x',t)\in \bar G_r$, 
$$|D_{x'}^lv(x',t)|\le B_0B^ll!,$$ 
for some positive constants $B_0$ and $B$. We denote by $v\in \mathcal A(\bar G_r)$. We note 
that the constants $B_0$ and $B$ are allowed to depend on $v$. It is easy to verify that, 
if $v\in \mathcal A(\bar G_r)$, then $D_{x'}^\tau v\in \mathcal A(\bar G_r)$, for any $\tau\ge 0$. 

\begin{cor}\label{cor-Linear-TangentialAnalyticity}
Under the assumptions of Theorem \ref{thrm-TangentialAnalyticity}, there holds, for any $r\in (0,1)$,  
\begin{equation}\label{eq-TangentialAnalyticity-u}
\frac{u}{t^2}, \, \frac{\partial_tu}{t},\, \partial_t^2u\in \mathcal A(\bar G_{r}).\end{equation}
\end{cor} 

\begin{proof} Fix an $r_0\in (1/2, 1)$. Under the assumptions of Theorem \ref{thrm-TangentialAnalyticity}, 
we have \eqref{eq-LinearAnalyticity-Tangential1}-\eqref{eq-LinearAnalyticity-Tangential3} in $G_R$,  
for some  positive constants $R\in (0,1/2)$, $B_0$ and $B$, independent of $l$.
By taking $r=R/2$, we have, 
for any $(x',t)\in G_r$ and any nonnegative integer $l$, 
\begin{align*}
|t^{-2}D^l_{x'} u(x',t)|&\leq B_0(r^{-1}B)^{l} l!,\\
|t^{-1}\partial_t D^l_{x'} u(x',t)|&\leq B_0 (r^{-1}B)^{l}l!,\\
|\partial_t^2 D^l_{x'} u(x',t)|&\leq B_0(r^{-1}B)^{l}l!.
\end{align*}
Similar estimates also hold for any $(x',t)\in B'_r(x'_0)\times (0,r)$, 
with $x_0'$ an arbitrary point in $B'_{r_0}$. These estimates hold in 
$B'_{r_0}\times (r,r_0)$ by the interior analyticity. Therefore, we have the desired result. 
\end{proof} 

\section{The Analyticity of Coefficients}\label{sec-AnalyticCoefficients} 

In this section,  we prove the analyticity of coefficients in the expansions near 
the boundary. The crucial step is to prove that the coefficient of the first nonlocal term is analytic. 

We start with the equation \eqref{eq-QuasiMain} and assume we can write it in the form 
\begin{align}\label{gtt}
u_{tt}+p\frac{u_t}{t}+q \frac{u}{t^2}=F, 
\end{align}
where $p$ and $q$ are constants and  $F$ is a function in $x', t$ and 
\be\label{gtt-Coefficients}
\frac{u}{t}, u_t, \frac{D_{x'} u}{t}, D_{x'}u_t, 
D^2_{x'}u, \frac{u^2}{t^3},\frac{uu_t}{t^2}, \frac{u_t^2}{t}.\ee
In the applications later on, $F$ is smooth in all of its arguments.

In the following, we denote by $\prime$ the derivative with respect to $t$. This 
should not be confused with $x'$, the first $n-1$ coordinates of the point. 

Throughout this section, we assume that $t^{\underline{m}}$ and $t^{\overline{m}}$ are 
solutions of the linear homogeneous equation corresponding to \eqref{gtt}; namely, 
\be\label{eq-Assumption_m1}
p=1-(\underline{m}+\overline{m}), \quad q=\underline{m}\cdot\overline{m}.\ee
We always assume that $\underline{m}$ and $\overline{m}$ are integers and satisfy 
\be\label{eq-Assumption_m2}\underline{m}\le 0,\quad \overline{m}\ge 3.\ee
We have the following simple result concerning the solution of \eqref{gtt}. 

\begin{lemma}\label{lemma-SolutionODE} Let $u$ be a solution of 
\eqref{gtt} in $G_r$ satisfying 
\be\label{eq-ODE-Assumption} 
t^{-\underline{m}}u\to 0\quad\text{as }t\to 0.\ee
Then,
\begin{align}\label{eq-LinearODE-Solution-0}\begin{split}
u(x', t)&=\left[u(x',r)r^{-\overline{m}}
+\frac{r^{\underline{m}-\overline{m}}}{\overline{m}-\underline{m}}
\int_0^r s^{1-\underline{m}}F(x',s)ds\right] t^{\overline{m}}\\
&\qquad-\frac{1}{\overline{m}-\underline{m}} 
t^{\underline{m}}\int_0^t s^{1-\underline{m}}F(x',s) ds\\
&\qquad -\frac{1}{\overline{m}-\underline{m}}
t^{\overline{m}}\int_t^r s^{1-\overline{m}}F(x',s) ds.\end{split}
\end{align}
\end{lemma}

Let $u\in C^1(\bar G_R)\cap C^{\infty}(G_R)$ be a solution of \eqref{eq-QuasiMain} and satisfy 
\eqref{eq-QuasiCondition_1}
and \eqref{eq-QuasiCondition_2}. 
Then, $u$ admits a formal expansion of the form 
\begin{equation}\label{eq-FormalExpansion}u(x', t)=\sum_{i=2}^{\overline{m}-1}c_i(x') t^i
+\sum_{i=\overline{m}}^\infty \sum_{j=0}^{N_i} c_{i, j}(x')t^{i}(\log t)^j.\end{equation}
In \cite{HanJiang2014}, we discussed the regularity of coefficients and the estimates of remainders. 
In particular, if $A_{ij}, P, Q$ and $N$ are smooth in its arguments, then all coefficients $c_i$ and $c_{i,j}$ 
are smooth functions of $x'$. In the next result, we prove that all coefficients are analytic 
if $A_{ij}, P, Q$ and $N$ are analytic in its arguments,

\begin{theorem}\label{lemma-LinearAnalyticity-FirstNonlocal}
Let $F$ in \eqref{gtt} be analytic in all of its arguments given by \eqref{gtt-Coefficients} and  
$\underline{m}$ and $\overline{m}$ be integers satisfying 
\eqref{eq-Assumption_m1} and \eqref{eq-Assumption_m2}.
Suppose that $u\in C^1(\bar G_1)\cap C^{\infty}(G_1)$ is a solution of \eqref{gtt} satisfying 
\eqref{eq-TangentialAnalyticity-u}. 
Then, all $c_i$ and $c_{i,j}$ are analytic in $B_1'$. 
\end{theorem}

\begin{proof} The coefficients in \eqref{eq-FormalExpansion} can be divided into two groups. The 
first group consists of $c_2, \cdots, c_{\overline{m}-1}$ and $c_{\overline{m},1}$, and the second group 
consists of the rest. Any coefficients in the first group 
can be expressed in terms of $c_2$ and its derivatives as well as $A_{ij}, P, Q$, and $N$, 
and any coefficients in the second group 
can be expressed in terms of $c_2, c_{\overline{m}, 0}$ 
and their derivatives as well as $A_{ij}, P, Q$, and $N$. See \cite{HanJiang2014} for details. 
We only need to prove $c_2$ and $c_{\overline{m},0}$ are analytic in $x'$. 

We first note 
$$c_2(x')=-\frac{N}{2A_{nn}+2P+Q}\bigg|_{(x',0)}.$$ 
Hence, the analyticity of $c_2$ follows from that of $A_{nn}, P, Q$ and $N$. 
Alternatively, we note 
$$c_2(x')=\lim_{t\to0}\frac{u(x',t)}{t^2}.$$
Then, $c_2$ is analytic in $x'$ since $u/t^2\in \mathcal A(\bar G_{r})$ by \eqref{eq-TangentialAnalyticity-u}. 
It remains to prove that $c_{\overline{m},0}$ is analytic in $x'$. To this end, 
we recall an integral expression of $c_{\overline{m}, 0}$. 

We now write \eqref{gtt}  as 
\begin{align}\label{eq-LinearODE-0}
u^{\prime\prime}+p\frac{u^\prime}{t}+q\frac{u}{t^2}=F\quad\text{in }G_1, 
\end{align}
and set $p_0=p, q_0=q, u_0=u, F_0=F$, 
$\underline{m}_0=\underline{m}$ and 
$\overline{m}_0=\overline{m}$. Set, for $l\ge 1$  inductively, 
\be\label{eq-definition-u-l} u_l=u_{l-1}^\prime-\frac{2u_{l-1}}{t}.\ee 
Then, 
\begin{align}\label{eq-LinearODE-l}
u_l^{\prime\prime}+p_l\frac{u_l^\prime}{t}+q_l\frac{u_l}{t^2}=F_l\quad\text{in }G_1, 
\end{align}
where, inductively,  
$$p_l=p_{l-1}+2, \quad q_l=p_{l-1}+q_{l-1},$$
and 
$$F_l=\partial_tF_{l-1}.$$
A simple calculation yields
$$p_l=2l+p,\quad q_l=l^2+(p-1)l+q,$$
and 
$$F_l=\partial_t^lF.$$
We note that
\be\label{eq-ExpressionPQ}p_l=1+2l-(\underline{m}+\overline{m}),
\quad q_l=(\underline{m}-l)(\overline{m}-l),\ee
and the general solutions of the homogeneous linear equation 
corresponding to \eqref{eq-LinearODE-l} are spanned by 
$t^{\underline{m}-l}$ and $t^{\overline{m}-l}$.

The solution $u$ can be expressed in terms of  $u_l$. To see this, 
we first rewrite \eqref{eq-definition-u-l} as 
\begin{equation}\label{eq-Iteration_u-l}\frac{u_l}{t^2}=\partial_t\left(\frac{u_{l-1}}{t^2}\right).\end{equation}
Then, inductively, 
$$\frac{u_l}{t^2}=\partial_t^l\left(\frac{u}{t^2}\right).$$
To proceed, we set
\begin{equation}\label{eq-Definition-v}
v_l=\frac{u_{l}}{t^2}=\partial_t^l\left(\frac{u}{t^2}\right).\end{equation}
Then, $v_l=v_{l-1}^\prime$. 
%Next, under the assumption that 
%$$\frac{u_0}{t^2}, \frac{u_1}{t^2}, \cdots, \frac{u_l}{t^2}$$ 
%are continuous up to $\Sigma_1$, 
If $v_0, v_1, \cdots, v_{l-1}$ are continuous in $\bar G_1$ and $v_l$ is integrable
in $G_1$, a successive integration yields 
\begin{align}\label{eq-Expression-u}\begin{split} 
u(x',t)&=v_0(x',0)t^2+v_1(x',0)t^3+\cdots+\frac{1}{(l-1)!}v_{l-1}(x',0)t^{l+1}\\
&\qquad +t^2\int_0^t\int_0^{s_1}\cdots\int_0^{s_{l-1}}v_{l}(x',s_{l})ds_{l}\cdots ds_2ds_1.
\end{split}\end{align}

For any $l\ge 0$, we now apply Lemma \ref{lemma-SolutionODE} to \eqref{eq-LinearODE-l}. 
Specifically, we replace $\underline{m}$ and $\overline{m}$ 
by $\underline{m}-l$ and $\overline{m}-l$
in \eqref{eq-LinearODE-Solution-0} and obtain 
\begin{align}\label{eq-LinearODE-Solution-l}\begin{split}
u_l(x', t)&=\left[u_l(x',r)r^{l-\overline{m}}
+\frac{r^{\underline{m}-\overline{m}}}{\overline{m}-\underline{m}}
\int_0^r s^{l+1-\underline{m}}F_l(x',s)ds\right] t^{\overline{m}-l}\\
&\qquad-\frac{1}{\overline{m}-\underline{m}} 
t^{\underline{m}-l}\int_0^t s^{l+1-\underline{m}}F_l(x',s) ds\\
&\qquad -\frac{1}{\overline{m}-\underline{m}}
t^{\overline{m}-l}\int_t^r s^{l+1-\overline{m}}F_l(x',s) ds.\end{split}
\end{align}
Now, we take  $l=\overline{m}-2$. Then, 
\begin{align*}%\label{eq-LinearODE-Solution-l}\begin{split}
u_{\overline{m}-2}(x', t)&=\left[u_l(x',r)r^{-2}
+\frac{r^{\underline{m}-\overline{m}}}{\overline{m}-\underline{m}}
\int_0^r s^{\overline{m}-\underline{m}-1}F_{\overline{m}-2}(x',s)ds\right] t^{2}\\
&\qquad-\frac{1}{\overline{m}-\underline{m}} 
t^{\underline{m}-\overline{m}+2}\int_0^t s^{\overline{m}-\underline{m}-1}F_{\overline{m}-2}(x',s) ds\\
&\qquad -\frac{1}{\overline{m}-\underline{m}}
t^{2}\int_t^r s^{-1}F_{\overline{m}-2}(x',s) ds.%\end{split}
\end{align*}
By setting 
$$\widetilde F_{\overline m-2}(x',t)=F_{\overline m-2}(x',t)-F_{\overline m-2}(x',0),$$
we have 
\begin{align*}
u_{\overline{m}-2}(x',t)=c_{2,1}(x') t^2\log t+c_{2,0}(x') t^2+ \widetilde w_{\overline{m}-2}(x',t),\end{align*}
where 
\begin{align}\label{eq-LinearODE-u-(n-2)}\begin{split}
c_{2,1}(x')&=\frac{1}{\overline{m}-\underline{m}}F_{\overline{m}-2}(x',0),\\
c_{2,0}(x')&=u_{\overline{m}-2}(x',r)r^{-2}
+\frac{r^{\underline{m}-\overline{m}}}{\overline{m}-\underline{m}}
\int_0^r s^{\overline{m}-\underline{m}-1}F_{\overline{m}-2}(x',s)ds\\
&\qquad-\frac{1}{(\overline{m}-\underline{m})^2}F_{\overline{m}-2}(x',0)
-\frac{\log r}{\overline{m}-\underline{m}}F_{\overline{m}-2}(x',0)\\
&\qquad-\frac{1}{\overline{m}-\underline{m}}
\int_0^r s^{-1}\widetilde F_{\overline{m}-2}(x',s)ds,
\end{split}\end{align}
and $\widetilde w_{\overline{m}-2}$ is the higher order term, which does not play any role in the 
present proof. 
Then, $c_{\overline{m},0}$ in \eqref{eq-FormalExpansion} %the expansion of $u$ %in \eqref{eq-LinearExpansion-m} 
is a linear combination of 
$c_{2,0}, c_{2,1}$ in \eqref{eq-LinearODE-u-(n-2)}. See \cite{HanJiang2014}, Lemma 5.2 in particular, for details. 
In the following, we will prove 
that $c_{2,0}, c_{2,1}$ in \eqref{eq-LinearODE-u-(n-2)} are analytic in $\bar B'_{r}$, for any $r\in (0,1)$. 

We first prove by induction, for $k=1, \cdots, \overline{m}-2$,  
\begin{equation}\label{eq-LinearAnalyticity-Induction}
\partial^k_t\left(\frac ut\right),\, \partial_t^{k+1}u\in \mathcal A(\bar G_{r}),\end{equation}
and 
\begin{equation}\label{eq-LinearAnalyticity-Induction2}
\partial^k_t\left(\frac {u^2}{t^3}\right),\, 
\partial^k_t\left(\frac {uu_t}{t^2}\right),\, 
\partial^k_t\left(\frac {u_t^2}{t}\right) \in \mathcal A(\bar G_{r}).\end{equation}
First, Corollary \ref{cor-Linear-TangentialAnalyticity} implies \eqref{eq-LinearAnalyticity-Induction}
and \eqref{eq-LinearAnalyticity-Induction2}
for $k=1$. We assume \eqref{eq-LinearAnalyticity-Induction} 
and \eqref{eq-LinearAnalyticity-Induction2} hold for $k=1, \cdots, l$, for some 
$l\le \overline{m}-3$. Since $F$ is analytic in $x'$, $t$ and quantities in 
\eqref{gtt-Coefficients}, 
then $F_l=\partial_t^lF\in \mathcal A(\bar G_{r})$. 
By \eqref{eq-LinearODE-Solution-l}, we get
$$\frac{u_l}{t^2},\, \frac{u_l'}{t},\, u_l''\in \mathcal A(\bar G_{r}).$$
It is essential here to assume $l\le\overline{m}-3$, because of the last integral in \eqref{eq-LinearODE-Solution-l}. 
With $v_l$ given by \eqref{eq-Definition-v}, we have 
$$v_l, \, tv_l', \, t^2v_l''\in \mathcal A(\bar G_{r}).$$
By \eqref{eq-Expression-u}, we have 
$$u(x',t)=v_0(x',0)t^2+v_1(x',0)t^3+\cdots+\frac{1}{(l-1)!}v_{l-1}(x',0)t^{l+1}+t^2R_l(x',t),$$
where
$$R_l(x',t)=\int_0^t\int_0^{s_1}\cdots\int_0^{s_{l-1}}v_{l}(x',s_{l})ds_{l}ds_{l-1}\cdots ds_1.$$
Note 
$$
\aligned\partial_t^{l+2}u&=t^2\partial_t^{l+2}R_l
+2(l+2)t\partial_t^{l+1}R_l+(l+2)(l+1)\partial_t^{l}R_l\\
&=t^2v_{l}^{\prime\prime}
+2(l+2)tv_{l}^\prime+(l+2)(l+1)v_{l}.
\endaligned$$
Then,  
$\partial_t^{l+2}u\in \mathcal A(\bar G_{r})$. 
Similarly, we have $\partial_t^{l+1}(u/t)\in \mathcal A(\bar G_{r})$. 
This is \eqref{eq-LinearAnalyticity-Induction} for $k=l+1$. We can also 
conclude \eqref{eq-LinearAnalyticity-Induction2} for $k=l+1$. 

By taking $k=\overline{m}-2$ in \eqref{eq-LinearAnalyticity-Induction}
and \eqref{eq-LinearAnalyticity-Induction2}, we conclude $F_{\overline{m}-2}\in \mathcal A(\bar G_{r})$. 
In particular, $F_{\overline{m}-2}(\cdot, 0)$ is analytic in $\bar B'_{r}$, and hence
$c_{2,1}$  in \eqref{eq-LinearODE-u-(n-2)} is analytic in $\bar B'_{r}$.

We now proceed to prove that $c_{2,0}$  in \eqref{eq-LinearODE-u-(n-2)} is analytic in $\bar B'_{r}$. 
By the expression of $c_{2,0}$ in \eqref{eq-LinearODE-u-(n-2)}, it suffices to prove 
that 
\begin{equation}\label{eq-LinearAnalycity-Difference}
\int_0^r s^{-1}\big[F_{\overline{m}-2}(x',s)-F_{\overline{m}-2}(x',0)\big]ds\end{equation}
is an analytic function of $x'$ in $B_{r}'$. 
By \eqref{eq-Definition-v} and  \eqref{eq-LinearODE-Solution-l} for $l=\overline{m}-2$, 
we have 
\begin{align*}v_{\overline{m}-2}(x', t)&=\left[u_{\overline{m}-2}(x',r)r^{-2}
+\frac{r^{\underline{m}-\overline{m}}}{\overline{m}-\underline{m}}
\int_0^r s^{\overline{m}-\underline{m}-1}F_{\overline{m}-2}(x',s)ds\right] \\
&\qquad-\frac{1}{\overline{m}-\underline{m}} 
t^{\underline{m}-\overline{m}}\int_0^t s^{\overline{m}-\underline{m}-1}F_{\overline{m}-2}(x',s) ds\\
&\qquad -\frac{1}{\overline{m}-\underline{m}}
\int_t^r s^{-1}F_{\overline{m}-2}(x',s) ds.\end{align*}
We note that the dominant term is the last integral. By the simple integral 
$\int_t^rs^{-1}ds=\log r-\log t$, we obtain, for any nonnegative integer $l$ and any 
$(x',t)\in \bar G_r$, 
$$|D_{x'}^lv_{\overline{m}-2}(x', t)|\le B_0B^ll!\log t^{-1},$$
for some positive constants $B_0$ and $B$. Moreover, a straightforward 
calculation yields 
$$|t^2D_{x'}^lv'_{\overline{m}-2}(x', t)|
+|tD_{x'}^lv_{\overline{m}-2}(x', t)|
+\left|\int_0^tD_{x'}^lv_{\overline{m}-2}(x', s)ds\right|\le B_0B^ll!t\log t^{-1}.$$
By \eqref{eq-Expression-u} with $l=\overline{m}-2$, we have 
$$u(x',t)=P_{\overline{m}-1}(x',t)+t^2R_{\overline{m}}(x', t),$$
where 
$$\aligned 
P_{\overline{m}-1}(x', t)&=v_0(x',0)t^2+v_1(x',0)t^3+\cdots
+\frac{1}{(\overline{m}-3)!}v_{\overline{m}-3}(x',0)t^{\overline{m}-1},\\
R_{\overline{m}}(x', t)&=\int_0^t\int_0^{s_1}\cdots\int_0^{s_{\overline{m}-3}}
v_{\overline{m}-2}(x',s_{\overline{m}-2})ds_{\overline{m}-2}\cdots ds_2ds_1.\endaligned$$
Then, 
$$\partial_t^{\overline{m}-1}u=\partial_t^{\overline{m}-1}P_{\overline{m}-1}
+t^2v_{\overline{m}-2}'+2(m-1)tv_{\overline{m}-2}
+(\overline{m}-1)(\overline{m}-2)\int_0^tv_{\overline{m}-2}(x',s)ds.$$
Note that $\partial_t^{\overline{m}-1}P_{\overline{m}-1}$ is a function of $x'$. Then, 
for any nonnegative integer $l$ and any 
$(x',t)\in \bar G_r$, 
$$\big|D^l_{x'}\big[\partial_t^{\overline{m}-1}u(x',t)-\partial_t^{\overline{m}-1}u(x',0)\big]\big|
\le B_0B^ll!t\log t^{-1}.$$
Similar estimates hold for $\partial_t^{\overline{m}-2}(u/t)$, $\partial_t^{\overline{m}-2}(u^2/t^3)$, 
$\partial_t^{\overline{m}-2}(uu_t/t^2)$ and $\partial_t^{\overline{m}-2}(u_t^2/t)$. 
Then, the expression of $F_{\overline{m}-2}$ implies 
$$\big|D^l_{x'}\big[F_{\overline{m}-2}(x',t)-F_{\overline{m}-2}(x',0)\big]\big|
\le B_0B^ll!t\log t^{-1}.$$
Therefore, 
\begin{align*}
&\, \left|D^l_{x'} \int_0^r s^{-1}\big[F_{\overline{m}-2}(x',s)-F_{\overline{m}-2}(x',0)\big] ds\right|\\
\le&\,\int_0^r s^{-1}\big|D^l_{x'} \big[F_{\overline{m}-2}(x',s)-F_{\overline{m}-2}(x',0)\big]\big| ds\\
\leq&\,  \int_0^r s^{-1}B_0B^l l! s \log s^{-1}  ds = \widehat B_0B^l l!. 
\end{align*}
This implies that the function in \eqref{eq-LinearAnalycity-Difference} is analytic in $B_\rho'$ for some 
$\rho<1/2$. 
\end{proof}

\section{The Analyticity along the Normal direction}
\label{Sec-NormalAnalyticity}

In this section, we prove the convergence of our boundary expansion. We 
adapt the methods by  Nirenberg \cite{Nirenberg1972}, 
by Kichenassamy and Littman \cite{KL:1}, \cite{KL:2}, and by Kichenassamy
\cite{Kichenassamy2004Adv}. We adopt the norm used by Nirenberg. 

We consider the equation \eqref{gtt}, with $F$ a function of 
$x'$, $t$ and those in 
\eqref{gtt-Coefficients}. We will prove that under appropriate assumptions, solutions 
are analytic in $x'$, $t$ and $t\log t$.  The main result in this paper is the following theorem. 

\begin{theorem}\label{Thm-MainThm}
Let $F$ in \eqref{gtt} be analytic in all of its arguments given by $x'$, $t$ and those in 
\eqref{gtt-Coefficients} and  let
$\underline{m}$ and $\overline{m}$ be integers satisfying 
\eqref{eq-Assumption_m1} and \eqref{eq-Assumption_m2}.
Suppose that $u\in C^1(\bar G_1)\cap C^{\infty}(G_1)$ is a solution of \eqref{gtt} satisfying 
\eqref{eq-TangentialAnalyticity-u}. 
Then, $u$ is analytic in $x'$, $t$ and $t\log t$, and the series in \eqref{eq-FormalExpansion} 
converges uniformly to $u$ in $G_{1/2}$. 
\end{theorem}

We first present a brief description of the proof.  For $k\ge \overline{m}$, set 
$$u_k(x', t)=\sum_{i=2}^{\overline{m}-1}c_i(x') t^i
+\sum_{i=\overline{m}}^k \sum_{j=0}^{N_i} c_{i, j}(x')t^{i}(\log t)^j.$$
Then, $u_k$ can be considered as a partial sum of the series \eqref{eq-FormalExpansion}.
Our goal is to prove that $u$ is analytic in $x'$, $t$ and $t\log t$ and that $u_k$ converges to $u$ 
uniformly. We can write $u_k$ as 
$$u_k(x', t)=\sum_{i=2}^{\overline{m}-1}c_i(x') t^i
+\sum_{i=\overline{m}}^k \sum_{j=0}^{N_i} c_{i, j}(x')t^{i-j}(t\log t)^j.$$
In other words, $u_k$ is the expansion of $u$ with respect to $t$ and $t\log t$ up to order $k$. 
In the proof below, we will construct another sequence $\{v_k\}$, analytic in $x'$, $t$ and $t\log t$ 
and converging uniformly to $v$, and we will prove that $u=v$. The first element $v_{\overline{m}}$ is 
given by $u_{\overline{m}+1}$ without the $t^{\overline{m}+1}$-term.

\begin{proof} %By Corollary \ref{cor-Linear-TangentialAnalyticity}, $u$ is analytic in $x'$. 
%We will prove that $u$ is analytic in $t$ and $t\log t$. 
The proof is quite long and is divided into two steps after the initial setup. 

For any function $v=v(x',t)$, we set 
\begin{equation}\label{eq-components}V=\bigg(\frac{v}{t}, v', \frac{D_{x'} v}{t}, D_{x'}v', 
D^2_{x'}v, \frac{v^2}{t^3},\frac{vv'}{t^2}, \frac{v'^2}{t}\bigg),\end{equation}
and write 
$$F=F(x',t,V).$$
We assume that there exist constants $M>0$ and $R\in (0,1)$, such that, for any 
$(x',t,V)$ with $|x'|+t+|V|<R$, 
$$|F(x',t,V)|\leq {M}\bigg[1-\frac{1}{R}\big(|x'|+t+|V|\big)\bigg]^{-1}.$$ 
We set 
$$L_0v=v^{\prime\prime}
+p\frac{v^{\prime}}{t}+q\frac{v}{t^2}.$$

{\it Step 1. We prove that \eqref{gtt} admits a unique solution  satisfying 
\eqref{eq-TangentialAnalyticity-u}.} Let $u$ be given as in the statement of Theorem \ref{Thm-MainThm}. 
Then, $u$ is analytic in $x'$ by Corollary \ref{cor-Linear-TangentialAnalyticity}. 
We extend arguments $x^\prime$ of all functions to the complex field. 
%$u$ is analytic in $x'$ and the series in \eqref{eq-FormalExpansion} 
%converges uniformly to $u$ in $G_{1/2}$.}
Set
\begin{align}\label{eq-vn}\begin{split}
{v}_{\overline{m}}(x',t)&=\sum_{i=2}^{\overline{m}-1}c_i(x^\prime) t^i
+c_{\overline{m},1}(x^\prime)t^{\overline m} \log t+c_{\overline{m}}(x^\prime) t^{\overline{m}}\\
&\qquad+\sum_{j=1}^{N_{\overline{m}+1}}c_{\overline{m}+1,j}(x^\prime)t^{\overline{m}+1} (\log t)^j,
\end{split}\end{align}
where $c_i$ and $c_{i,j}$ are as in \eqref{eq-FormalExpansion}. 
Note that ${v}_{\overline{m}}$ is holomorphic in $x'$ and 
is the expansion of $u$ before the term $t^{\overline{m}+1}$, i.e., 
$$|u-v_{\overline{m}}|\le Ct^{\overline{m}+1}.$$ 
%For convenience, we set ${v}_{\overline{m}-1}={v}_{\overline{m}}$. 
Inductively, for $k\geq \overline{m}+1$, we define $w_k$ and $v_k$ by 
\begin{align}\label{eq-gk}
w_{k}= \frac{1}{{\overline{m}}-\underline{m}}t^{\underline{m}}\int_0^t \rho^{1-\underline{m}} {F}_k
d\rho- \frac{1}{{\overline{m}}-\underline{m}}t^{{\overline{m}}}\int_0^t \rho^{1-{\overline{m}}} {F}_k d\rho,\end{align}
and 
$${v}_k={v}_{k-1}+w_k={v}_{\overline{m}}
+\sum_{i={\overline{m}}+1}^k w_i,$$
where
\begin{align*}
{F}_{{\overline{m}}+1}=
F({x^\prime},t,V_{\overline{m}})
-L_0{v}_{\overline{m}},
\end{align*}
and, for $k\geq \overline{m}+2$,
\begin{align*}
{F}_k= F(x^\prime,t,V_{k-1})
-F(x^\prime,t,V_{k-2}).\end{align*}
Here, 
$$V_k=\bigg(\frac{v_k}{t}, v_k', \frac{D_{x'} v_k}{t}, D_{x'}v_k', 
D^2_{x'}v_k, \frac{v_k^2}{t^3},\frac{v_kv_k'}{t^2}, \frac{v_k'^2}{t}\bigg).$$
It is easy to see that  $w_k$ is a solution of the equation $L_0w_k=F_k$. 
%with $|w_k|\le Ct^{{\overline{m}}+1}$.
Hence, 
\begin{align}\label{Eq-Series}
L_0v_k=
F({x^\prime},t,V_{k-1}).
\end{align}
%and 
%$$|{v}_k-{v}_{\overline{m}}|\le Ct^{\overline{m}+1}.$$ 
%Note that is $v_k$ analytic in $x'$. 

In the following, we will prove that 
${v}_k$ is holomorphic in $x^\prime$ and converges uniformly in a fixed region $|x'|+t<r$. To this end, we 
need to introduce appropriate domains and norms. 

Let $w=w(x')$ be a holomorphic function in $B_1'\subset\mathbb C^{n-1}$. Define, 
for any $r\in ( 0,1)$, 
$$\|w\|_r=\sup_{|x'|<r}|w(x')|.$$
By the usual estimate for derivatives of holomorphic functions, we have, 
for any $0< r'<r$ and any $\alpha=1, \cdots, n-1$,
\begin{equation}\label{eq-EstimateHolomorphicDerivative}
\|\partial_{\alpha}w\|_{r'}\le \frac{1}{r-r'}\|w\|_r.\end{equation}
This estimate will be used repeatedly in the following. 

With $a_0$ a positive constant to be determined later, define inductively, for any $k\ge 0$,  
\begin{equation}\label{eq-induction-a}a_{k+1}=a_k(1-(k+1)^{-2}).\end{equation}
Then, 
$$a\equiv \lim_{k\rightarrow \infty} a_k=a_0 \prod_{k=0}^\infty (1-(k+1)^{-2})>0.$$ 
For some fixed $s_0>0$, set 
$$\Omega_k=\left\{(x',t):\, |x'|+\frac{t}{a_k}<s_0, \, t>0\right\}.$$
For any function $w(x^\prime,t)$ defined in $\Omega_k$,  holomorphic 
in $x'$ and continuous in $t$, we write $w(t)=w(\cdot, t)$ and define 
\begin{align*}
M_k[w]=\sup_{\substack{0< s <s_0\\0<t<a_k (s_0-s)}}
\bigg[\frac{\|w(t)\|_s}{t^{\overline{m}-1}} \bigg(\frac{a_k(s_0-s)}{t}-1\bigg)\bigg].\end{align*}
We also define 
$$\Omega_\infty=\left\{(x',t):\, |x'|+\frac{t}{a}<s_0, \, t>0\right\},$$
and 
\begin{align*}
M_\infty[w]=\sup_{\substack{0< s <s_0\\0<t<a (s_0-s)}}
\bigg[\frac{\|w(t)\|_s}{t^{\overline{m}-1}} \bigg(\frac{a(s_0-s)}{t}-1\bigg)\bigg].\end{align*}

Inductively, we will prove that there exist positive constants $A$ and $s_0$ such that, 
for  any $k\ge \overline{m}+1$, any $0<s<s_0$ and $0<t<a_k (s_0-s)$, 
\begin{align}\label{eq-upbound}
&\|V_{k-1}(t)\|_s \leq \frac{R}{4},\\
\label{eq-holomorphic}&{w_k}\text{ is holomorphic in $x^\prime$ in $\Omega_k$},\\
\label{eq-normbound}
&M_k \big[\frac{w_k}{t}\big],\, M_k [w_k^\prime], \,M_k [D_{x^\prime} w_{k}]\leq \frac{A}{2^k}.
\end{align}

For $k=\overline{m}+1$, 
we can set $A$ large and $s_0$ small such that \eqref{eq-upbound}, \eqref{eq-holomorphic}
and \eqref{eq-normbound} hold. We can set $s_0^{{\overline{m}}-2} A$ small for later purposes.
For \eqref{eq-upbound} with $k=\overline{m}+1$, we require a stronger estimate 
$$\|V_{\overline{m}}(t)\|_s \leq \frac{R}{64}.$$
We assume that \eqref{eq-upbound}, \eqref{eq-holomorphic}
and \eqref{eq-normbound} hold for $k-1$, for some $k\ge \overline{m}+2$, and proceed to consider for $k$.

%To illustrate ideas more clearly, we first assume that $F$ depends only on the first five components  
%in \eqref{eq-components}. 

First, we prove \eqref{eq-upbound}. There are eight components in $V_{k-1}$.  
%All functions appeared are holomorphic in $x^\prime$. We can discuss the first $3$ terms 
%in the same way as in Section 2 in \cite{Nirenberg1972}.
We consider $D^2_{x^\prime} v_{k-1}$ for an illustration. For each $i=\overline{m}+1, \cdots, k-1$,  $w_{i}$
is holomorphic in $x'$ for any $(x',t)\in\Omega_{i}$. 
Hence, we can apply \eqref{eq-EstimateHolomorphicDerivative}
to each $w_i$. 
For any $0<s<s_0$ and $0<t<a_{k}(s_0-s)$, we set 
$$\tau_i=\frac{1}{2}\big(s_0+s-\frac{t}{a_i}\big).$$
Then, 
\begin{align*}
\|D^2_{x^\prime} w_{i}(t)\|_s&\le \frac{\|D_{x^\prime} w_{i}(t)\|_{\tau_{i}}}{\tau_{i}-s}
\le  \frac{t^{{\overline{m}}-1}M_{i}[D_{x^\prime}w_{i}]}{(\tau_{i}-s)(a_{i}(s_0-\tau_{i})/t-1)}\\
&\leq  \frac{4a_it^{{\overline{m}}-1}}{t(a_{i}(s_0-s)/t-1)^2}M_{i}[D_{x^\prime}w_{i}]\\
&\leq 4a_0 t^{{\overline{m}}-2} A\big(\frac{a_{i}}{a_{i+1}}-1\big)^{-2} 2^{-i},\end{align*}
and hence
\begin{align*}
\|D^2_{x^\prime} v_{k-1}(t)\|_s &\leq \|D^2_{x^\prime} v_{{\overline{m}}}(t)\|_s
+\sum_{i={\overline{m}}+1}^{k-1}  ||D^2_{x^\prime} w_{i}(t)||_s\\
&\leq  \frac{R}{64}+Ca_0 s_0^{{\overline{m}}-2}A\sum_{
i={\overline{m}}+1}^{k} \big(\frac{a_{i}}{a_{i+1}}-1\big)^{-2}2^{-i}\leq \frac{R}{32},
\end{align*}
since $s_0^{{\overline{m}}-2} A$ is small. We can discuss other components in $V_{k-1}$ similarly. 

Second, we prove \eqref{eq-holomorphic}. 
Recall that $w_k$ is given by  \eqref{eq-gk}. Note that 
\begin{equation}\label{eq-Fkmvt}
F_k=\int_0^1D_VF(x',t,\theta V_{k-1}+(1-\theta)V_{k-2})d\theta\cdot (V_{k-1}-V_{k-2}),\end{equation}
where, for $s_0<R/4$, 
\begin{align*}
|D_VF|
\leq {M}\bigg[1- \frac1R\big(s_0+ a_0 s_0+\frac{R}{4}\big)\bigg]^{-1}
\leq C,
\end{align*}
and
\begin{equation}\label{eq-components-W}\begin{split}
V_{k-1}-V_{k-2}&=\bigg(\frac{w_{k-1}}{t}, w_{k-1}', \frac{D_{x'} w_{k-1}}{t}, D_{x'}w_{k-1}', 
D^2_{x'}w_{k-1}, \\
&\qquad\frac{v_{k-1}^2}{t^3}-\frac{v_{k-2}^2}{t^3},\frac{v_{k-1}v_{k-1}'}{t^2}-\frac{v_{k-2}v_{k-2}'}{t^2}, 
\frac{v_{k-1}'^2}{t}-\frac{v_{k-2}'^2}{t}\bigg).\end{split}\end{equation}
%since we assumed that $F$ depends only on the first five components  
%in \eqref{eq-components}. 
For the last three components, we have
\begin{align*}
\frac{v_{k-1}v^\prime_{k-1}}{t^2}-\frac{v_{k-2}v^\prime_{k-2}}{t^2}
=\frac{v_{k-1}}{t^2}w^\prime_{k-1}+ \frac{v_{k-2}^\prime}{t}\frac{w_{k-1}}{t},
\end{align*}
and similar identities for the other two components. 
Note that ${v_{k-1}}/{t^2}, {v_{k-2}^\prime}/{t}$ are bounded. 
So, we need to prove $t^{-{\overline{m}}} w_{k-1}$ and $t^{1-{\overline{m}}}{w_{k-1}^\prime}$ 
are holomorphic in $x^\prime$ in $\Omega_{k}$. 
%The problem is to deal with the $t^{-{\overline{m}}}, t^{1-{\overline{m}}}$ factor. 
For any $0<s<s_0$ and $0<t<a_k(s_0-s)$, we take 
$$\tau=\frac{1}{2}\big(s_0+s-\frac{t}{a_{k-1}}\big),$$ 
and get
\begin{align*}
t^{-{\overline{m}}} \|D_{x^\prime}w_{k-1}(t)\|_s &\leq Ct^{-{\overline{m}}} 
\frac{\|w_{k-1}(t)\|_\tau}{\tau-s}
\leq C\frac{4a_{k-1} M_{k-1}[\frac{w_{k-1}}{t}]}{t(a_{k-1}(s_0-s)/t-1)^2}\\
&\leq C\frac{4a_0  2^{-k+1}A t}{(a_{k-1}-a_k)^2(s_0-s)^2}.
\end{align*}
We can discuss other terms similarly. 
%{\color{red} (There are many other terms in $D_{x'}w_k$.)}

Last, we prove \eqref{eq-normbound}.  For any $0<s<s_0$ and $0<t<a_k(s_0-s)$, we have 
\begin{align*}
\frac{\|w_k(t)\|_s}{t^{\overline{m}}}&\leq 
\frac{1}{{\overline{m}}-\underline{m}}t^{\underline{m}-{\overline{m}}}
\int_0^t \rho^{1-\underline{m}} \|{F}_k\|_s
d\rho+\frac{1}{{\overline{m}}-\underline{m}}\int_0^t \rho^{1-{\overline{m}}} \|{F}_k\|_s  d\rho.\end{align*}
By \eqref{eq-Fkmvt}, 
$$|F_k|\le C|V_{k-1}-V_{k-2}|.$$
The integrals split to several parts. 
We first consider $w_{k-1}/t$. By 
\begin{align*}
\rho^{1-{\overline{m}}} \frac{\|w_{k-1}(\rho)\|_s}{\rho} 
\leq  \frac{1}{a_{k-1}(s_0-s))/ \rho-1}M_{k-1}\big[\frac{w_{k-1}}{t}\big],
\end{align*}
we have
\begin{align*}
\int_0^t \rho^{1-{\overline{m}}} \frac{\|w_{k-1}(\rho)\|_s}{\rho} 
d\rho 
&\leq  \frac{t}{a_{k-1}(s_0-s)/t-1}M_{k-1}\big[\frac{w_{k-1}}{t}\big]\\
&\leq \frac{a_k s_0 }{a_k(s_0-s)/t-1}M_{k-1}\big[\frac{w_{k-1}}{t}\big].
\end{align*}
We can discuss $w^\prime_{k-1}$ similarly.
Next, we consider $D_{x'}^2w_{k-1}$. 
For each $\rho\in (0,t)$, we take $s(\rho)<s_0-\frac{\rho}{a_{k-1}}$ to be fixed. Then, 
\begin{align}\label{eq-estD2g}\begin{split}
\rho^{1-\overline{m}} \|D^2_{x^\prime}w_{k-1}(\rho)\|_{s} & \leq 
\rho^{1-\overline{m}} \frac{\|D_{x^\prime}w_{k-1}(\rho)\|_{s(\rho)}}{s(\rho)-s} \\
&\leq 
\frac{M_{k-1}[D_{x^\prime}w_{k-1}]}{(s(\rho)-s)(a_{k-1}(s_0-s(\rho))/\rho-1)}.
\end{split}\end{align}
By taking 
$$s(\rho)=\frac{1}{2}\big(s_0+s-\frac{\rho}{a_{k-1}}\big),$$ 
%and note 
%\begin{align*}
%M_{k}[D_{x^\prime}w_{k-1}]\leq M_{k-1}[D_{x^\prime}w_{k-1}].\end{align*}
%Then, 
we have
\begin{align*}
\int_0^t \rho^{1-\overline{m}} \|D^2_{x^\prime}w_{k-1}(\rho)\|_{s} d\rho
\leq \frac{Ca_k}{a_{k-1}(s_0-s)/t-1}M_{k-1}[D_{x^\prime}w_{k-1}].
\end{align*}
Similar estimates hold for other integrals. By requiring $Ca_0$ to be small, we have
\begin{align*}
M_k\big[\frac{w_k}{t}\big]\leq \frac{1}{6}(M_{k-1}[w_{k-1}]+M_{k-1}[w_{k-1}^\prime]+M_{k-1}[
D_{x^\prime}w_{k-1}]).
\end{align*}
Similar estimates hold for $w_{k}^\prime$. Hence, 
\begin{align*}
M_{k}\big[\frac{w_{k}}{t}\big],\, M_k[w_k^\prime] \leq\frac{A}{2^{k}}.
\end{align*}
Next, we note
\begin{align*}
D_{x^\prime} w_{k}(t) =\int_0^t D_{x^\prime} w_{k}^\prime(\rho) d\rho.\end{align*}
For $t<a_{k}(1-s)$, we have, similarly as in \eqref{eq-estD2g},
\begin{align*}
M_{k}[D_{x^\prime} w_{k}] \leq \frac{1}{2}M_k[w_{k}^\prime]\leq \frac{A}{2^{k}}.
\end{align*}
We hence have \eqref{eq-normbound}.

In conclusion, by \eqref{eq-normbound}, we have, for any $k\ge \overline{m}+1$, 
$$M_\infty \big[\frac{w_k}{t}\big],\, M_\infty [w_k^\prime], \,M_\infty [D_{x^\prime} w_{k}]\leq \frac{A}{2^k}.$$
There exists a  function $v$, holomorphic in $x'$ for any $(x',t)\in \Omega_\infty$, such that 
$${v}_k\to v,\quad v_k'\to v', \quad D_{x'}v_k\to D_{x'}v\quad\text{in }\Omega_\infty\text{ as }k\to\infty.$$ 
We also have, for $|x^\prime|< s$ and $t< a(s_0-s)$,
\begin{align*}
\bigg|t^{1-{\overline{m}}}\big(\frac{a(s_0-s)}{t}-1\big)\frac{{v}-{v}_{\overline{m}}}{t}\bigg|
\leq\sum_{k={\overline{m}}+1}^\infty M_k\big[\frac{w_k}{t}\big]\leq A,
\end{align*}
and 
\begin{align*}
\bigg|t^{1-{\overline{m}}}\big(\frac{a(s_0-s)}{t}-1\big)D_{x^\prime}({v}-{v}_{\overline{m}})\bigg|
\leq\sum_{k={\overline{m}}+1}^\infty M_k[D_{x^\prime}{w_k}]\leq A.
\end{align*}
Hence,  
$v$ satisfies 
\begin{align}\label{Eq-Series1}
L_0v=
F({x^\prime},t,V),
\end{align}
and 
\begin{equation}\label{Eq-Condition1}|{v}-{v}_{\overline{m}}|\le Ct^{\overline{m}+1}.\end{equation}
Moreover, 
$t^{-\overline{m}}(v-{v}_{\overline{m}})$ is holomorphic in $x'$ in $\Omega_\infty$. 
%This is in fact implied by the proof of Lemma \ref{lemma-LinearAnalyticity-FirstNonlocal}, 
%where we essentially proved that the integrand of remainders of expansions of $u$ 
%up to any order $k$ is tangentally analytic.
 
%Especially both $v$, $\overline{v}$ are of form $\overline{v}_{\overline{m}}+O(t^{{\overline{m}}+1})$.
Last, we prove $u=v$. Note that  $u$ also satisfies \eqref{Eq-Series1} 
and \eqref{Eq-Condition1}, with $v$ and $V$ replaced by $u$ and $U$, respectively. 
Set $w= u-v$. Then,  
$$L_0w=F(x',t,U)-F(x',t,V)\quad\text{in }\Omega_\infty,$$ 
and 
$$|w|\le Ct^{\overline{m}+1}.$$ 
We can repeat the above iteration and have $M_k[u-v]=M_k[w_k]\to 0$, as $k\to \infty$. 
This implies $u=v$. 

\smallskip

{\it Step 2. We prove that $u$ is analytic in $t, t\log t$.}
We treat $t$ and $t\log t$ as two independent variables and set
$$T=t,\quad S=t\log t.$$ 
For a function $u=u(x', t, t\log t)=u(x', T, S)$, we have 
$$t\partial_tu=t\partial_Tu+t(\log t+1)\partial_Su=T\partial_Tu+(T+S)\partial_Su.$$
Set 
$$\Lambda=T\partial_T+(T+S)\partial_S.$$
Then, 
$$t\partial_tu=\Lambda u.$$
Next, we extend arguments $x', T, S$ into the complex field. 
Since the complexified $t$ or $T$ cannot be the upper bounds in the integral,
we need to make a change of variables so that $t$ or $T$ appears in the integrands. 
A simple substitution yields 
$$\int_0^tu(x', s, s\log s)ds=\int_0^1 Tu(x', \rho T, \rho(\log \rho)T+\rho S)d\rho.$$

We now take the same sequences ${v}_k$ and ${w}_k$ as in Step 1 and treat them 
as functions of $x', T$ and $S$. 
We start by writing ${v}_{\overline{m}}$ in the form
\begin{align}\label{eq-vn1}\begin{split}
{v}_{\overline{m}}&=c_2(x^\prime) T^2+\cdots+c_{{\overline{m}}-1}(x^\prime)T^{{\overline{m}}-1}
+c_{{\overline{m}},1}(x^\prime)T^{{\overline{m}}-1}S\\
&\qquad+c_{\overline{m}}(y^\prime) T^{\overline{m}}
+\sum_{j=1}^{N_{\overline{m}+1}}c_{\overline{m}+1,j}(x^\prime)T^{\overline{m}+1-j} S^j.\end{split}\end{align}
For $k\geq {\overline{m}}+1$, define $w_k(x^\prime, T, S)$ and $v_k(x^\prime, T, S)$ inductively by 
\begin{align}\begin{split}\label{eq-gk1}
w_{k}&= \frac{1}{{\overline{m}}-\underline{m}}T^{2}\int_0^1  
\rho^{1-\underline{m}} {F}_k(x^\prime, \rho T, \rho(\log \rho)T+\rho S)
d\rho\\
&\qquad - \frac{1}{{\overline{m}}-\underline{m}}T^{2}\int_0^1 
\rho^{1-{\overline{m}}} {F}_k(x^\prime, \rho T, \rho(\log \rho)T+\rho S)d\rho,
\end{split}\end{align}
and 
$${v}_k={v}_{\overline{m}}
+\sum_{i={\overline{m}}+1}^k w_i,$$
where 
\begin{align*}
{F}_{{\overline{m}}+1}(x',T,S)= 
F\big({x^\prime},T,V_{\overline{m}}(x',T,S)\big)
-\frac{1}{T^2}(\Lambda^2\overline{v}_{\overline{m}}
+(p-1)\Lambda\overline{v}_{\overline{m}}+q \overline{v}_{\overline{m}})(x',T,S),
\end{align*}
and, for $k\geq {\overline{m}}+2$,
\begin{align*}
{F}_k=F\big(x^\prime,T, V_{k-1}(x',T,S)\big)
-F\big(x^\prime,T,V_{k-2}(x',T,S)\big).
\end{align*}
Here, 
$$V_k=\bigg(\frac{v_k}{T}, \frac{\Lambda v_k}{T}, \frac{D_{x'} v_k}{T}, \frac{D_{x'}\Lambda v_k}{T}, 
D^2_{x'}v_k,\frac{v_k^2}{T^3},\frac{v_k\Lambda v_k}{T^3}, \frac{(\Lambda v_k)^2}{T^3}\bigg).$$
Then, $w_k$ is a solution of 
\begin{align*}
\Lambda^2w_k
+(p-1)\Lambda w_k+q w_k=T^2{F}_k.
\end{align*}

In the following, we will prove that ${v}_k$ is holomorphic in $x^\prime, T, S$ and 
converges uniformly in a fixed region $|x'|+|T|+|S|<r$. 

We fix an
arbitrary $\theta\in (0,1)$ and let $\{a_k\}$ be introduced as in \eqref{eq-induction-a}. For convenience, 
we write $u(T, S)=u(\cdot, T, S)$ and set 
$$\delta=|T|+\theta|S|.$$
We define 
\begin{align*}
\|u(T, S)\|_r=\sup_{|x'|<r} |u(x', T, S)|,
\end{align*}
and 
\begin{align*}
M_k[u]=\sup_{\substack{0< s <s_0, T\neq0\\
\delta<a_k (s_0-s)}} \bigg[\frac{\|u(T, S)\|_s}{|T|^{\overline{m}-1}}
\bigg(\frac{a_k(s_0-s)}{\delta}-1\bigg)\bigg].\end{align*}
Set  
$${\Omega}_{k}=\{(x^\prime,T, S):\, |x'|+\delta/a_k<s_0 \}.$$ 
We also define 
$$\Omega_\infty=\left\{(x^\prime,T, S):\, |x'|+\delta/a<s_0\right\},$$
and 
\begin{align*}
M_\infty[w]=\sup_{\substack{0<s <s_0, T\neq0\\
\delta<a (s_0-s)}} \bigg[\frac{\|u(T, S)\|_s}{|T|^{\overline{m}-1}}
\bigg(\frac{a(s_0-s)}{\delta}-1\bigg)\bigg].\end{align*}

Inductively, we will prove that there exist positive constants $A$ and $s_0$ such that, 
for  any $k\ge \overline{m}+1$, any $0<s<s_0$ and $\delta<a_k (s_0-s)$, 
\begin{align}\label{eq-upbound1}
&\|V_{k-1}(T,S)\|_s \leq \frac{R}{4},\\
\label{eq-holomorphic1}&\frac{w_k}{T^{\overline{m}}}, \frac{\Lambda w_k}{T^{\overline{m}}}
\text{ are holomorphic in $(x^\prime, T, S)\in\Omega_k$},\\
\label{eq-normbound1}&M_k [\frac{w_k}{T}], M_k [\frac{\Lambda w_k}{T}], M_k [D_{x^\prime} w_{k}]\leq \frac{A}{2^k}.
\end{align}

For $k=\overline{m}+1$, 
we can set $A$ large and $s_0$ small such that \eqref{eq-upbound1}, \eqref{eq-holomorphic1}
and \eqref{eq-normbound1} hold. We can set $s_0^{{\overline{m}}-2} A$ small for later purposes.
For \eqref{eq-upbound1}, we require a stronger estimate 
$$\|V_{\overline{m}}(T,S)\|_s \leq \frac{R}{64}.$$
We assume that \eqref{eq-upbound1}, \eqref{eq-holomorphic1}
and \eqref{eq-normbound1} hold for $k-1$, for some $k\ge \overline{m}+2$, and proceed to consider for $k$.

First, we prove \eqref{eq-upbound}. There are eight components in $V_{k-1}$.  
%All functions appeared are holomorphic in $x^\prime$. We can discuss the first $3$ terms 
%in the same way as in Section 2 in \cite{Nirenberg1972}.
We consider $D^2_{x^\prime} v_{k-1}$ for an illustration. For each $i=\overline{m}+1, \cdots, k-1$,  $w_{i}$
is holomorphic in $\Omega_{i}$. 
%Hence, we can apply \eqref{eq-EstimateHolomorphicDerivative} to each $w_i$. 
For any $0<s<s_0$ and $\delta<a_{k}(s_0-s)$, we set 
$$\tau_i=\frac{1}{2}\big(s_0+s-\frac{\delta}{a_i}\big).$$
Then, 
\begin{align*}
\|D^2_{x^\prime} w_{i}(T,S)\|_s&\le \frac{\|D_{x^\prime} w_{i}(T,S)\|_{\tau_{i}}}{\tau_{i}-s}
\le  \frac{|T|^{{\overline{m}}-1}M_{i}[D_{x^\prime}w_{i}]}{(\tau_{i}-s)(a_{i}(s_0-\tau_{i})/\delta-1)}\\
&\leq  \frac{4a_i|T|^{{\overline{m}}-1}}{\delta(a_{i}(s_0-s)/\delta-1)^2}M_{i}[D_{x^\prime}w_{i}]\\
&\leq 4a_0 |T|^{{\overline{m}}-2} A\big(\frac{a_{i}}{a_{i+1}}-1\big)^{-2} 2^{-i},\end{align*}
and hence
\begin{align*}
\|D^2_{x^\prime} v_{k-1}(T,S)\|_s &\leq \|D^2_{x^\prime} v_{{\overline{m}}}(T,S)\|_s
+\sum_{i={\overline{m}}+1}^{k-1}  \|D^2_{x^\prime} w_{i}(T,S)\|_s\\
&\leq  \frac{R}{64}+Ca_0 s_0^{{\overline{m}}-2}A\sum_{
i={\overline{m}}+1}^{k} \big(\frac{a_{i}}{a_{i+1}}-1\big)^{-2}2^{-i}
\leq \frac{R}{32},
\end{align*}
since $s_0^{{\overline{m}}-2} A$ is small. We can discuss other components in $V_{k-1}$ similarly. 

Second, we prove \eqref{eq-holomorphic1}. 
Recall that $w_k$ is given by  \eqref{eq-gk1}. Note that 
\begin{equation}\label{eq-Fkmvt1}F_k(x',T,S)=\int_0^1D_VF(x',T,\theta V_{k-1}+(1-\theta)V_{k-2})d\theta
\cdot (V_{k-1}-V_{k-2}),\end{equation}
where, for $s_0<R/4$, 
\begin{align*}
|D_VF|
\leq {M}\bigg[1- \frac1R\big(s_0+ a_0 s_0+\frac{R}{4}\big)\bigg]^{-1}
\leq C,
\end{align*}
and
\begin{align*}%\label{eq-components-W1}\begin{split}
V_{k-1}-V_{k-2}&=\bigg(\frac{w_{k-1}}{T}, \frac{\Lambda w_{k-1}}{T}, \frac{D_{x'} w_{k-1}}{T}, 
\frac{D_{x'} \Lambda w_{k-1}}{T}, D^2_{x'}w_{k-1},\\
&\qquad\frac{v_{k-1}^2}{T^3}-\frac{v_{k-2}^2}{T^3},\frac{v_{k-1}\Lambda v_{k-1}}{T^3}-\frac{v_{k-2}\Lambda v_{k-2}}{T^3}, 
\frac{(\Lambda v_{k-1})^2}{T^3}-\frac{(\Lambda v_{k-2})^2}{T^3}\bigg).%\end{split}  
\end{align*}
In view of \eqref{eq-gk1} and the induction that $\frac{w_{k-1}}{T^{\overline{m}}}$ and 
$\frac{\Lambda w_{k-1}}{T^{\overline{m}}}$ are holomorphic in $(x', T, S)\in \Omega_{k-1}$, 
we need to analyze the impact of the factor $\rho^{1-{\overline{m}}}$ in the second term of \eqref{eq-gk1}. 
Divided by $T^{{\overline{m}}}$, the expression
\begin{align*}
T^{2-\underline{m}} \int_0^1\rho^{1-\overline{m}} {F}_k (x^\prime, \rho T, \rho(\log \rho)T+\rho S)d\rho
\end{align*}
is holomorphic in $x^\prime, T$ and $ S$, since 
$V_{k-1}-V_{k-2}$ 
in \eqref{eq-Fkmvt1} and their derivatives can absorb the $\rho^{1-{\overline{m}}}$ factor. For example,
corresponding to the first component in $V_{k-1}-V_{k-2}$, we write 
\begin{align*}
&T^{2-{\overline{m}}} \int_0^1\rho^{1-{\overline{m}}} 
\frac{w_{k-1} }{\rho T}(x^\prime, \rho T, \rho(\log \rho)T+\rho S) d\rho\\
&\qquad=  T\int_0^1\frac{ w_{k-1}}{ (\rho T)^{{\overline{m}}}} (x^\prime, \rho T, \rho(\log \rho)T+\rho S )d\rho,
\end{align*}
which is holomorphic in $x^\prime, T$ and $S$, since the integrand is holomorphic in its arguments. 
Corresponding to the second component in $V_{k-1}-V_{k-2}$, we write 
\begin{align*}
&  T^{2-{\overline{m}}} \int_0^1\rho^{1-{\overline{m}}} 
\frac{\Lambda w_{k-1} }{\rho T}(x^\prime, \rho T, \rho(\log \rho)T+\rho S) d\rho\\
&\qquad= \int_0^1 T D_T \frac{w_{k-1}}{(\rho T)^{{\overline{m}}-1}}(x^\prime, \rho T, \rho(\log \rho)T+\rho S)\\
&\qquad\qquad+ ({\overline{m}}-1)\frac{w_{k-1}}{(\rho T)^{{\overline{m}-1}}}(x^\prime, \rho T, \rho(\log \rho)T+\rho S)\\
&\qquad\qquad+(\log \rho)TD_S \frac{w_{k-1}}{(\rho T)^{{\overline{m}}-1}}(x^\prime, \rho T, \rho(\log \rho)T+\rho S)\\
&\qquad\qquad+SD_S \frac{w_{k-1}}{(\rho T)^{{\overline{m}-1}}}(x^\prime, \rho T, \rho(\log \rho)T+\rho S) d\rho,
\end{align*}
which is holomorphic in $x^\prime, T$ and $S$. We can discuss other terms similarly. 

Last, we prove \eqref{eq-normbound1}. Set
$$\delta(\rho)=\rho|T|+\theta|\rho(\log \rho)T+\rho S|,$$ 
and 
$$h(\rho)=\rho-\theta\rho\log \rho.$$ 
Then, $h$ is an increasing function in $(0,1)$ and hence $h(\rho)\le h(1)=1$ for any $\rho\in (0,1)$. 
It is easy to check, for any $\rho\in (0,1)$, 
$$\delta(\rho)\le h(\rho)\delta\le \delta.$$
For any $0<s<s_0$ and $\delta<a_k(s_0-s)$, we have 
\begin{align*}
\frac{|w_k(T,S)|}{|T|^{\overline{m}}}&\leq 
\frac{1}{\overline{m}-\underline{m}}|T|^{2-\overline{m}}\int_0^1  
\rho^{1-\underline{m}} |{F}_k(x^\prime, \rho T, \rho(\log \rho)T+\rho S)|
d\rho\\
&\qquad + \frac{1}{\overline{m}-\underline{m}}|T|^{2-\overline{m}}\int_0^1 
\rho^{1-\overline{m}} |{F}_k(x^\prime, \rho T, \rho(\log \rho)T+\rho S)|d\rho.\end{align*}
By \eqref{eq-Fkmvt1}, we have 
\begin{align*}
|{F}_k| \leq C|V_{k-1}-V_{k-2}|.
\end{align*}
The integrals split to several parts. First, we have 
\begin{align*}
&\big\|T^{2-{\overline{m}}}\int_0^1 \rho^{1-{\overline{m}}}
\frac{w_{k-1}}{\rho T}
(x^\prime, \rho T, \rho(\log \rho)T+\rho S)d\rho\big\|_s\\
&\qquad\leq  M_{k-1}\big[\frac{w_{k-1}}{T}\big]|T|\int_0^1 
\frac{d\rho}{a_{k-1}(s_0-s)/\delta(\rho)-1}\\
&\qquad\leq  \frac{|T|}{a_{k-1}(s_0-s)/\delta-1}M_{k-1}[\frac{w_{k-1}}{T}]\\
&\qquad\leq \frac{a_k s_0 }{a_k(s_0-s)/\delta-1}M_{k-1}\big[\frac{w_{k-1}}{T}\big].
\end{align*}
We can discuss $ \frac{\Lambda w_{k-1}}{T}$ in a similar way.
Next, we consider $D_{x'}^2w_{k-1}$. 
For each $\rho\in (0,1)$, we take $s<s(\rho)<s_0-\frac{\rho}{a_{k-1}}$ to be fixed. Then, 
\begin{align*}
&\big\|T^{2-\overline{m}}\int_0^1 \rho^{1-\overline{m}}
D^2_{x^\prime}w_{k-1}
 (x^\prime, \rho T, \rho(\log \rho)T+\rho S)d\rho\big\|_{s}\\
 &\quad\leq |T|^{2-\overline{m}}\int_0^1 \rho^{1-\overline{m}}
\frac{\|D_{x^\prime}w_{k-1}(x^\prime, \rho T, \rho(\log \rho)T+\rho S)\|_{s(\rho)}}{s(\rho)-s}d\rho\\
 &\quad\leq C|T| M_{k-1}[D_{x^\prime}w_{k-1}]\cdot I,\end{align*}
where
\begin{align*}I= \int_0^1 
\frac{d\rho}{[s(\rho)-s][a_{k-1}(s_0-s(\rho))/\delta(\rho)-1]}.\end{align*}
Set 
$$s(\rho)=\frac{1}{2}\big(s_0+s-\frac{h(\rho)}{a_{k-1}}\big).$$
Then, 
\begin{align*}I&\le \int_0^1 
\frac{d\rho}{[s(\rho)-s][a_{k-1}(s_0-s(\rho))/\delta h(\rho)-1]}\\
&\le \int_0^1 
\frac{4a_{k-1}d\rho}{\delta h(\rho)[a_{k-1}(s_0-s(\rho))/\delta h(\rho)-1]^2}.\end{align*}
Introduce a new variable 
$$\tau=\frac{\delta h(\rho)}{a_{k-1}(s_0-s)}.$$
Then, with $\delta <a_{k-1}(s_0-s)$, 
\begin{align*}
I&\leq \int^{ \delta/ a_{k-1}(s_0-s)}_0 \frac{4a_{k-1}d\tau}{\delta\tau(\tau^{-1}-1)^2
(1-\theta-\theta\log \rho) } \\
&\leq \frac{4a_{k-1}}{\delta(1-\theta)} \int^{ \delta/ a_{k-1}(s_0-s)}_0 \frac{\tau d\tau}{(1-\tau)^2}  \\
&\leq \frac{4a_{k-1}}{\delta(1-\theta)} \int^{ \delta/ a_{k-1}(s_0-s)}_0 \frac{d\tau}{(1-\tau)^2}  \\
&= \frac{4a_{k-1}}{\delta(1-\theta)} \frac{1}{a_{k-1}(s_0-s)/\delta-1}.
\end{align*}
Hence, with $|T|\leq \delta$, we get 
\begin{align*}
&\big\|T^{2-\overline{m}}\int_0^1 \rho^{1-\overline{m}}
D^2_{x^\prime}w_{k-1}
 (x^\prime, \rho T, \rho(\log \rho)T+\rho S)d\rho\big\|_{s}\\
&\qquad \leq \frac{4a_{k-1}}{\delta(1-\theta)} M_{k-1}[D_{x^\prime}w_{k-1}] \frac{1}{a_{k-1}(s_0-s)/\delta-1}.
\end{align*}
Similar estimates hold for other terms. By taking $Ca_0$ small, we have
\begin{align*}
M_k[\frac{w_k}{T}]\leq \frac{1}{6}\big(M_{k-1}[\frac{w_{k-1}}{T}]+M_{k-1}[\frac{\Lambda w_{k-1}}{T}]+M_{k-1}[
D_{x^\prime}w_{k-1}]\big).
\end{align*}
Next, we consider $ \Lambda w_{k}$. Note
\begin{align*}
D_\rho{F}_k(x^\prime, \rho T, \rho(\log \rho)T+\rho S)
=\rho^{-1}\Lambda{F}_k(x^\prime, \rho T, \rho(\log \rho)T+\rho S),
\end{align*}
and hence
\begin{align*}
&\int_0^1 \rho^{1-\overline{m}} \Lambda{F}_k(x^\prime, \rho T, \rho(\log \rho)T+\rho S)d\rho\\
&\qquad=\int_0^1 \rho^{2-\overline{m}} D_\rho {F}_k(x^\prime, \rho T, \rho(\log \rho)T+\rho S) d\rho.\end{align*}
Integrating by parts, we can estimate all terms similarly.
Hence, we conclude
$$M_k\big[\frac{w_k}{T}\big],\, M_k\big[\frac{\Lambda w_k}{T}\big]\leq \frac{A}{2^k}.$$
Next, note 
\begin{align*}
D_{x^\prime} w_{k}(x^\prime, T, S)&=\int_0^1 D_\rho D_{x^\prime} w_{k}(x^\prime, \rho T, \rho(\log \rho) T+\rho S) d\rho\\
&=\int_0^1 T D_{x^\prime} \frac{\Lambda w_{k}}{\rho T}(x^\prime, \rho T, \rho(\log \rho) T+\rho S)d\rho.\end{align*}
For $\delta<a_{k}(1-s)$, we have similarly
\begin{align*}
M_{k}[D_{x^\prime} w_{k}] \leq \frac{1}{2}M_{k}[\frac{\Lambda w_{k}}{T}]\leq \frac{A}{2^{k}}.
\end{align*}

Therefore, we conclude ${v}_k\to {v}$ in the norm $M_\infty$
and hence $v$ is holomorphic in $(x^\prime, T, S)\in \Omega_\infty$. Moreover, the Taylor series of $v$ 
in terms of $T$ and $S$ converges to $v$ uniformly for $|x'|+|T|+|S|<r$. 
By $u=v$ for $T=t$ and $S=t\log t$ and a comparison of coefficients, we obtain that 
$u$ is analytic in $x^\prime, t, t\log t$ and that the series in \eqref{eq-FormalExpansion} 
converges uniformly to $u$ for $|x'|+t<r/2$.
\end{proof}

Theorem \ref{thrm-Analyticity-VerticalGraph} follows easily from
Theorem \ref{Thm-MainThm}.

\section{The Loewner-Nirenberg Problem}\label{Sec-LN}

In this section, we discuss briefly the Loewner-Nirenberg problem. 

Let $\Omega$ be a bounded domain in $\mathbb R^n$, for some $n\ge 3$. Consider 
\begin{align}
\label{eq-LN-MainEq} \Delta u  &= \frac14n(n-2) u^{\frac{n+2}{n-2}} \quad\text{in }\Omega,\\
\label{eq-LN-MainCond}u&=\infty\quad\text{on }\partial \Omega.
\end{align}
We let $d(x)=\operatorname{dist}(x, \partial \Omega)$
be the distance of $x$ to the boundary $\partial\Omega$. 

Assume that $\Omega$ has a $C^{1,1}$-boundary. Loewner and Nirenberg \cite{Loewner&Nirenberg1974} proved that 
\eqref{eq-LN-MainEq} and \eqref{eq-LN-MainCond} admit a unique positive solution $u\in C^\infty(\Omega)$ and that 
there exists a constant $\mu>0$ such that, for any $x\in \Omega$ with $d(x)<\mu$,
\begin{equation}\label{eq-LN-basic}\big|d^{\frac{n-2}2}(x)u(x)-1\big|\le Cd(x),\end{equation}
where $C$ is a positive constant 
depending only on $n$ and the $C^{1,1}$-norm of $\partial\Omega$. 

Solutions of \eqref{eq-LN-MainEq}-\eqref{eq-LN-MainCond} 
are known to have formal expansions. 
In the case that $\Omega$ is a bounded smooth domain, $d$ is a smooth 
function near $\partial\Omega$. For each $x\in \Omega$ close to $\partial\Omega$, 
there exists a unique $z\in \partial\Omega$ such that $d(x)=|x-z|$. 
Then, a formal expansion of $u$ is given by
$$d^{-\frac{n-2}{2}}\bigg(1+\sum_{i=1}^{n-1}c_id^i
+\sum_{i=n}^\infty \sum_{j=0}^{N_i} c_{i, j}d^i(\log d)^j\bigg),$$
where $c_i$ and $c_{i,j}$ are smooth functions of $z\in \partial\Omega$,
and $N_i$ is a nonnegative constant depending on $i$, with $N_n=1$.
A formal calculation can only determine {\it finitely many terms 
in the formal expansion of} $u$ near $\partial\Omega$. 
In fact, the coefficients $c_1$, $\cdots, c_{n-1}$ and 
$c_{n,1}$
have explicit expressions in terms of principal curvatures of $\partial \Omega$ and 
their derivatives. 
For example, 
$$c_1=\frac{n-2}{4(n-1)}H,$$
and, for $n=3$, 
$$c_{3,1}=-\frac{1}{16}\big\{\Delta_{\partial \Omega} H+2H(H^2-K)\big\},$$
where $H$ and $K$ are the mean curvature and the Gauss curvature of $\partial\Omega$, 
respectively. We note that $c_{3,1}=0$ 
if and only if $\partial\Omega$ is a Willmore surface. 

Mazzeo \cite{Mazzeo1991} and 
Andersson, Chru\'sciel, and Friedrich \cite{ACF1982CMP} proved that solution $u$ of 
\eqref{eq-LN-MainEq}-\eqref{eq-LN-MainCond} is polyhomogeneous if $\Omega$ has a smooth boundary. 

To analyze behaviors of solutions near the boundary, we introduce a new function with the 
zero boundary value. Let $u\in C^\infty(\Omega)$ 
be a solution of \eqref{eq-LN-MainEq}-\eqref{eq-LN-MainCond}. Set 
\be\label{eq-LN-Relation}u=d^{-\frac{n-2}2}(1+v).\ee
If $\partial\Omega\in C^{1,1}$, then $v$ satisfies 
\begin{align}\label{eq-LN-NewMainEq}
\mathcal S(v)=0\quad\text{in }\Omega,
\end{align}
and, by \eqref{eq-LN-basic},  
\begin{align}\label{eq-LN-NewMainCond}|v|\le Cd\quad\text{in }\Omega,\end{align}
where 
\begin{align*}
\mathcal S(v)&=d^2\Delta v-(n-2)d\nabla d\cdot \nabla v-\frac12(n-2)d\Delta d(1+v)\\
&\qquad-\frac14n(n-2)\left[(1+v)^{\frac{n+2}{n-2}}-(1+v)\right].\end{align*}
In particular, $v$ is continuous up to the boundary and $v=0$ on $\partial\Omega$. 
We note that $\mathcal S$ is a semilinear elliptic operator, degenerate along $\partial\Omega$. 
We rewrite $\mathcal S$ as 
\begin{align}\label{eq-LN-NonlinearOperator}\begin{split}
\mathcal S(v)&=d^2\Delta v-(n-2)d\nabla d\cdot \nabla v-nv-\frac12(n-2)d\Delta dv-\frac12(n-2)d\Delta d\\
&\qquad-\frac14n(n-2)\left[(1+v)^{\frac{n+2}{n-2}}-1-\frac{n+2}{n-2}v\right].\end{split}\end{align}
In the expression of $\mathcal S$ in \eqref{eq-LN-NonlinearOperator}, the first four terms 
are linear in $v$, the fifth term is the nonhomogeneous term, and the final term is a nonlinear expression of $v$. 
Methods in \cite{HanJiang2014} and in this paper can be adapted to treat \eqref{eq-LN-NewMainEq}.

Let $k\ge n$ be an integer and set, for $z\in \partial\Omega$ and $d>0$,  
\begin{align}\label{eq-Sum-Intro}
S_k(z,d)=1+\sum_{i=1}^{n-1}c_i(z)d^i
+\sum_{i=n}^k \sum_{j=0}^{N_i} c_{i, j}(z)d^i(\log d)^j.
\end{align}
We point out that the highest order in the parenthesis is given by $d^k$. 
According to the pattern in this expansion, if we intend to continue to 
expand, the next term has an order of  
$d^{k+1}(\log d)^{N_{k+1}}$. 

Similarly as in \cite{HanJiang2014}, we have the
regularity and growth of the remainder $d^{\frac{n-2}2}u-S_k$ as well as 
the regularity of the coefficients $c_i$ and $c_{i,j}$. 

\begin{theorem}\label{thrm-LN-Main-Intro}
For some integer $k\ge n$ and some constant $\alpha\in (0,1)$, 
assume $\partial\Omega\cap B_R(z_0)$ is $C^{k+1, \alpha}$, 
for some $z_0\in \partial\Omega$ and $R>0$, 
and let $u\in C^{\infty}(\Omega\cap B_R(z_0))$ be 
a solution of  \eqref{eq-LN-MainEq}-\eqref{eq-LN-MainCond}. Then, 
there exist functions  $c_i$, $c_{i,j}\in C^{k-i, \epsilon}(\partial\Omega\cap B_R(z_0))$, 
for $i=1, \cdots, k$ and $j=0, 1, \cdots, N_i$,
and any $\epsilon\in (0,\alpha)$, 
such that, 
for $S_k$ defined as in \eqref{eq-Sum-Intro},
for any $m=0, 1, \cdots, k$, any $\epsilon\in (0,\alpha)$, and any $r\in (0, R)$, 
\begin{equation}\label{eq-MainRegularity}\partial_{d}^m \big(d^{\frac{n-2}2}u(x)-S_{k}(z,d)\big)\in 
C^{\epsilon}(\bar\Omega\cap B_r(z_0)),
\end{equation}
and, for any $x\in 
\Omega\cap B_{R/2}(z_0)$,  
\begin{equation}\label{eq-MainEstimate}\big|\partial_{d}^m \big(d^{\frac{n-2}2}u(x)-S_{k}(z,d)\big)\big|
\le C d^{k-m+\alpha},
\end{equation}
where $d=d(x)$, $z\in\partial\Omega$ is the unique point with 
$d(x)=|x-z|$ and $C$ is a positive constant depending only on $n$, $k$, $\alpha$, $R$,
the $L^\infty$-norm of $d^{\frac{n-2}2}u$ in $\Omega\cap B_R(z_0)$ and  
the $C^{k+2, \alpha}$-norm  of $\partial\Omega\cap B_R(z_0)$.  
\end{theorem} 
Here we have one loss of regularity, but we remark there is no regularity loss for  a similar result of function $d^\frac{n}{2}u(x)$.
Concerning the analyticity, we have the following result. 

\begin{theorem}\label{thrm-LN-Main-Analyticity}
Assume $\partial\Omega\cap B_R(z_0)$ is analytic,
for some $z_0\in \partial\Omega$ and $R>0$.
Let $u\in C^{\infty}(\Omega\cap B_R(z_0))$ be 
a solution of  \eqref{eq-LN-MainEq}-\eqref{eq-LN-MainCond}. Then, $u$ is analytic in $z$, $d$ and $d\log d$
in $\bar\Omega\cap B_{R/2}(x_0)$. Moreover, let 
$S_k$ be defined as in \eqref{eq-Sum-Intro} satisfying 
\eqref{eq-MainEstimate}. Then, 
$$S_{k}(z,d)\to d^{\frac{n-2}2}u(x)\quad\text{uniformly in }\Omega\cap B_{R/2}(x_0).$$
\end{theorem} 

\appendix

\section{Analyticity Estimates}\label{Subsec-Analyticity}

In this section, we present an analyticity type estimate for compositions of functions, 
which is due to Friedman.
The following result is essentially Lemma 1 in \cite{Friedman1958} with $M_l=l!$.

\begin{lemma}\label{lemma-Composition}
Let $\Omega$ be a domain in $\mathbb R^n$ and $p$ be a positive integer. Assume that $\Phi$ 
is a $C^p$-function in $ \Omega\times\mathbb R^N$ satisfying, for any 
$(x,y)\in \Omega\times\mathbb R^N$ and any nonnegative integers $j$ and $k$ 
with $j+k\le p$, 
\begin{equation}\label{eq-Composition1}
\left|\frac{\partial^{j+k}\Phi(x,y)}{\partial x^j\partial y^k}\right|
\le A_0A_1^jA_2^k(j-2)!(k-2)!,\end{equation}
for some positive constants $A_0$, $A_1$ and $A_2$. Then, there exist 
positive constants $B_0$, $\widetilde B_0$ and $B_1$, 
depending only on $n$, $N$, $A_0$, $A_1$ and $A_1$, such that, 
for any $C^p$-function $y=(y_1, \cdots, y_N): \Omega\to\mathbb R^N$, if 
for any $x\in \Omega$ and any nonnegative integer $k\le p$, 
\begin{equation}\label{eq-Composition2}
\sum_{i=1}^N|\partial^k_xy_i(x)|\le B_0B_1^{(k-2)^+}(k-2)!,\end{equation}
then, for any $x\in\Omega$, 
\begin{equation}\label{eq-Composition3}
|\partial_x^p[\Phi(x,y(x))]|\le \widetilde B_0B_1^{(p-2)^+}(p-2)!.\end{equation}
\end{lemma} 

\begin{proof} 
Set $t=x_1+\cdots+x_n$. Fix an $x\in\Omega$ and write $y=y(x)$. 
We will construct scalar-valued $C^p$-functions
$z(t)$ and $\Psi(t,z)$ such that, 
%$z$ controls $y_1, \cdots, y_N$ and $\Psi$ controls $\Phi$. 
for $k=1, \cdots, p$, 
\begin{equation}\label{eq-Composition4}
|\partial_{(x,y)}^k\Phi(x, y)|\le \partial^k_{(t,z)}\Psi(0,0),\end{equation}
and 
\begin{equation}\label{eq-Composition5} 
\sum_{i=1}^N|\partial_x^ky_i(x)|\le \frac{d^k}{dt^k}z(0).\end{equation} 
Also, we require $z(0)=0$. 

First, we note 
$$\partial_x^p[\Phi(x,y(x))]=\sum
\frac{\partial^{|\alpha_0|+k}\Phi}{\partial x^{\alpha_0}\partial y_{i_1}\cdots\partial y_{i_k}}(x,y)
\frac{\partial^{|\alpha_1|}y_{i_1}}{\partial x^{\alpha_1}}(x)\cdots
\frac{\partial^{|\alpha_k|}y_{i_k}}{\partial x^{\alpha_k}}(x),$$
where the summation is for $\alpha_0, \alpha_1, \cdots, \alpha_k\in \mathbb Z^n_+$ 
with $|\alpha_0|+|\alpha_1|+\cdots+|\alpha_k|=p$ and also for $i_1, \cdots, i_k$ from 1 to $N$. 
By \eqref{eq-Composition4} and \eqref{eq-Composition5}, we have 
$$|\partial_x^p[\Phi(x,y(x))]|\le\sum
\frac{\partial^{|\alpha_0|+k}\Psi}{\partial t^{|\alpha_0|}\partial z^k}(0,0)
\frac{d^{|\alpha_1|}z}{d t^{|\alpha_1|}}(0)\cdots
\frac{d^{|\alpha_k|}z}{d t^{|\alpha_k|}}(0).$$
Hence, 
\begin{equation}\label{eq-Composition6}
|\partial_x^p[\Phi(x,y(x))]|\le \frac{d^p}{dt^p}[\Psi(t,z(t))]\big|_{t=0}.\end{equation}

In view of \eqref{eq-Composition1}, we set 
\begin{equation}\label{eq-Composition7}
\Psi(t,z)=\Psi_1(t)\Psi_2(z),\end{equation}
where 
\begin{equation}\label{eq-Composition8}
\Psi_1(t)=\sum_{i=0}^p\frac{A_1^i(i-2)!}{i!}t^i,\end{equation}
and 
\begin{equation}\label{eq-Composition9}
\Psi_2(z)=A_0\sum_{i=0}^p\frac{A_2^i(i-2)!}{i!}z^i.\end{equation}
Then, \eqref{eq-Composition4} holds. 

Next, we set 
\begin{equation}\label{eq-Definition_z}
z(t)=B_0\left[t+\sum_{k=2}^p\frac{1}{k(k-1)}B_1^{k-2}t^k\right].\end{equation} 
By \eqref{eq-Composition2}, \eqref{eq-Composition5} holds. 

Now we start to estimate the right-hand side of \eqref{eq-Composition6}. 
We claim, for any $i=1, \cdots, p$, 
\begin{equation}\label{eq-Claim_z_i}
[z(t)]^i=B_0^i\left[t^i+\sum_{k=i+1}^pa_{i,k}B_1^{k-i-1}t^k\right]+O(t^{p+1}),\end{equation} 
where $a_{i,k}$ is a nonnegative constant satisfying, for $1\le i<k\le p$,  
\begin{equation}\label{eq-Claim_z_i_coefficients}
a_{i,k}\le \frac{3^{i-1}}{(k-i+1)(k-i)}.\end{equation}
To prove \eqref{eq-Claim_z_i} and \eqref{eq-Claim_z_i_coefficients}, 
we first note that \eqref{eq-Definition_z} implies 
\eqref{eq-Claim_z_i} with $i=1$ and the equality holds in \eqref{eq-Claim_z_i_coefficients}
with $i=1$. We assume that 
\eqref{eq-Claim_z_i} and \eqref{eq-Claim_z_i_coefficients} hold for some $i=1, \cdots, p-1$. 
Next, we consider $i+1$. A simple multiplication of  \eqref{eq-Definition_z} and 
\eqref{eq-Claim_z_i}  yields 
\begin{align*}
[z(t)]^{i+1}&=B_0^{i+1}\bigg[t^{i+1}+\sum_{k=2}^pB_1^{k-2}a_{1,k}t^{k+i}
+\sum_{l=i+1}^pB_1^{l-i-1}a_{i,l}t^{l+1}\\
&\qquad +\sum_{k=2, l=i+1}^pB_1^{k+l-i-3}a_{1,k}a_{i,l}t^{k+l}\bigg]+O(t^{p+1}).
\end{align*}
By a change of indices in summations, we have 
$$[z(t)]^{i+1}=B_0^{i+1}\bigg[t^{i+1}+\sum_{m=i+2}^pB_1^{m-i-2}a_{i+1,m}t^{m}\bigg]+O(t^{p+1}),$$
where, for $m=i+2, \cdots, p$,  
$$a_{i+1,m}=a_{1, m-i}+a_{i,m-1}+\operatorname{sgn}(m-i-2)B_1^{-1}
\sum_{\substack{k+l=m\\ k\ge2, l\ge i+1}}a_{1,k}a_{i,l}.$$
By \eqref{eq-Claim_z_i_coefficients}, we have 
$$a_{i+1,m}\le \frac{1+3^{i-1}}{(m-i)(m-i-1)}+\frac{3^{i-1}}{B_1}
\sum_{\substack{k+l=m\\ k\ge2, l\ge i+1}}\frac{1}{k(k-1)(l-i+1)(l-i)}.$$
Note $k\le 2(k-1)$ for $k\ge 2$ and $l-i+1\le 2(l-i)$ for $l\ge i+1$. Then, 
\begin{align*}\frac{1}{k(k-1)(l-i+1)(l-i)}&\le \frac{4}{k^2(l-i+1)^2}
=\frac{4}{(k+l-i+1)^2}\left[\frac1k+\frac{1}{l-i+1}\right]^2\\
&\le \frac{8}{(k+l-i+1)^2}\left[\frac1{k^2}+\frac1{(l-i+1)^2}\right].
\end{align*}
Hence, 
\begin{align*}
\sum_{\substack{k+l=m\\ k\ge2, l\ge i+1}}\frac{1}{k(k-1)(l-i+1)(l-i)}&
\le \frac{8}{(m-i+1)^2}\left[\sum_{k\ge 2}\frac{1}{k^2}+\sum_{l\ge i+1}\frac{1}{(l-i+1)^2}\right]\\
&\le \frac{16}{(m-i+1)^2}\left(\frac{\pi^2}{6}-1\right)\le  \frac{16}{(m-i+1)^2}. \end{align*}
A simple substitution yields 
$$a_{i+1,m}\le \frac{1+3^{i-1}}{(m-i)(m-i-1)}+\frac{3^{i-1}\cdot 16B_1^{-1}}{(m-i)(m-i-1)}.$$
If $B_1\ge 16$, then 
$$a_{i+1,m}\le \frac{3\cdot 3^{i-1}}{(m-i)(m-i-1)}=\frac{3^{i}}{(m-i)(m-i-1)}.$$
This proves \eqref{eq-Claim_z_i} and \eqref{eq-Claim_z_i_coefficients} for $i+1$. 

By \eqref{eq-Composition9}, we write 
$$\Psi_2(z)=A_0\left[1+A_2z+\sum_{i=2}^p\frac{A_2^i}{i(i-1)}z^i\right],$$
and hence 
$$\Psi_2(z(t))=A_0\left[1+A_2z(t)+\sum_{i=2}^p\frac{A_2^i}{i(i-1)}[z(t)]^i\right].$$
We claim, for $k=1, \cdots, p$, 
\begin{equation}\label{eq-DerivativePsi2}\frac{d^k}{dt^k}\Psi_2(z(t))|_{t=0}
\le \widehat B_0B_1^{(k-2)^+}(k-2)!,
\end{equation}
where 
$$\widehat B_0=A_0B_0(9A_2+A_2^2B_0).$$
First, we have 
$$\frac{d}{dt}\Psi_2(z(t))|_{t=0}=A_0A_2\frac{dz}{dt}(0)=A_0A_2B_0.$$
Then, \eqref{eq-DerivativePsi2} holds for $k=1$ since $\widehat B_0\ge A_0A_2B_0$. 
Next, for $k=2, \cdots, p$, 
$$\aligned \frac{d^k}{dt^k}\Psi_2(z(t))|_{t=0}
&=A_0A_2\frac{d^kz}{dt^k}(0)+A_0\sum_{i=2}^p\frac{A_2^i}{i(i-1)}\frac{d^k}{dt^k}[z(t)]^i|_{t=0}\\
&=A_0A_2B_0B_1^{k-2}(k-2)!+A_0\sum_{i=2}^{k-1}\frac{A_2^i}{i(i-1)}a_{i,k}B_0^iB_1^{k-i-1}k!\\
&\qquad+A_0A_2^kB_0^k(k-2)!.
\endaligned$$
Hence, by \eqref{eq-Claim_z_i_coefficients}, 
$$\aligned \frac{d^k}{dt^k}\Psi_2(z(t))|_{t=0}
&\le A_0B_0B_1^{k-2}(k-2)!\\
&\qquad\cdot\bigg[A_2+\sum_{i=2}^{k-1}\frac{3^{i-1}A_2^iB_0^{i-1}}{B_1^{i-1}}
\frac{k(k-1)}{i(i-1)(k-i+1)(k-i)}+\frac{A_2^kB_0^{k-1}}{B_1^{k-2}}\bigg].
\endaligned$$
We note, for $2\le i\le k-1$,  
$$\frac{k(k-1)}{i(i-1)(k-i+1)(k-i)}\le 8.$$
To prove this, we consider $2\le i\le k/2$ first and have 
$$\frac{k(k-1)}{(k-i+1)(k-i)}\le \frac{k(k-1)}{(k/2+1)k/2}\le\frac{4(k-1)}{k+2}\le 4.$$
For $k/2<i\le k-1$, we have $k\ge 3$ and 
$$\frac{k(k-1)}{i(i-1)}<\frac{k(k-1)}{k/2(k/2-1)}\le \frac{4(k-1)}{k-2}\le 8.$$
Therefore, 
$$\aligned\frac{d^k}{dt^k}\Psi_2(z(t))|_{t=0}
&\le A_0B_0B_1^{k-2}(k-2)!\\
&\qquad\cdot\left[A_2
+8A_2\sum_{i=2}^{k-1}\left(\frac{3A_2B_0}{B_1}\right)^{i-1}
+A_2^2B_0\left(\frac{A_2B_0}{B_1}\right)^{k-1}\right].\endaligned
$$
If $B_1\ge 6A_2B_0$, we obtain \eqref{eq-DerivativePsi2} for $k=2, \cdots, p$.

Next, by \eqref{eq-Composition8}, we have, for any $i=0, 1, \cdots, p$,  
\begin{equation}\label{eq-DerivativePsi1}
\frac{d^i\Psi_1}{dt^i}(0)=A_1^i(i-2)!.\end{equation}
With \eqref{eq-Composition7}, \eqref{eq-DerivativePsi2} and \eqref{eq-DerivativePsi1}, we have 
$$\aligned \frac{d^p}{dt^p}\Psi(t, z(t))|_{t=0}&=\sum_{i=0}^p
\frac{p!}{i!(p-i)!}\frac{d^i}{dt^i}\Psi_2(z(t))|_{t=0}\cdot \frac{d^{p-i}}{dt^{d-i}}\Psi_1|_{t=0}\\
&\le \sum_{i=0}^p\frac{p!}{i!(p-i)!}(p-i-2)!(i-2)! \widehat B_0 A_1^{p-i}B_1^{(i-2)^+}\\
&=\widehat B_0B_1^{p-2}(p-2)! \sum_{i=0}^p\frac{p(p-1)}{i!(p-i)!}(p-i-2)!(i-2)! \frac{A_1^{p-i}B_1^{(i-2)^+}}{B_1^{p-2}}.
\endaligned$$
We consider $p\ge 3$. By considering $i=0$ and $p$, $i=1$ and $p-1$, and $2\le i\le p-2$, we have 
$$\frac{p(p-1)}{i!(p-i)!}(p-i-2)!(i-2)! \le 8.$$
Therefore, 
$$\frac{d^p}{dt^p}\Psi(t,z(t))|_{t=0}\le\sum_{i=0}^p
\widehat B_0B_1^{p-2}(p-2)! 
\left[(A_1^2+3A_1)\left(\frac{A_1}{B_1}\right)^{p-2}+8\sum_{i=2}^p\left(\frac{A_1}{B_1}\right)^{p-i}\right].$$
By taking $B_1\ge 2A_1$, we have 
$$\frac{d^p}{dt^p}\Psi(t,z(t))|_{t=0}\le(A_1^2+3A_1+16)\widehat B_0B_1^{p-2}(p-2)!.$$
In summary, we take $B_1\ge \max\{16, 6A_2B_0, 2A_1\}$ and 
$$\widetilde B_0=(A_1^2+3A_1+16)\widehat B_0=A_0B_0(9A_2+A_2^2B_0)(A_1^2+3A_1+16),$$
and then have the desired result. 
\end{proof} 

\begin{remark}\label{rmk-Composition} Write $x=(x',x_n)$. In Lemma \ref{lemma-Composition}, 
if we assume \eqref{eq-Composition1} and 
\eqref{eq-Composition2} hold only for $D_{x'}$ instead of $D_x$, then 
\eqref{eq-Composition3} holds for $D_{x'}$.\end{remark}


\begin{thebibliography}{DG}

\bibitem{Anderson1982Invent} M. Anderson, \emph{Complete minimal varieties in 
hyperbolic space}, Invent. Math., 69(1982), 477-494. 

\bibitem{Anderson1983} M. Anderson, \emph{Complete minimal hypersurfaces in hyperbolic 
$n$-manifolds}, Comment. Math. Helv., 58(1983), 264-290. 

\bibitem{Anderson2003} M. Anderson, \emph{Boundary regularity, uniqueness and 
non-uniqueness for AH Einstein metrics on 4-manifolds}, 
Adv. Math., 179(2003), 205-249. 

\bibitem{ACF1982CMP} L. Andersson, P. Chru\'sciel, H. Friedrich, 
\emph{On the regularity of solutions to the Yamabe equation and the existence of 
smooth hyperboloidal initial data for Einstein�s field equations}, 
Comm. Math. Phys., 149(1992), 587-612. 

\bibitem{Biquad2010} 
O. Biquard, M. Herzlich, 
\emph{Analyse sur un demi-espace hyperbolique et 
poly-homogeneite locale},  arXiv:1002.4106.

\bibitem{ChengYau1980CPAM} S.-Y. Cheng, S.-T. Yau, 
\emph{On the existence of a complete K\"ahler metric on non-compact complex 
manifolds and the regularity of Fefferman's equation}, 
Comm. Pure Appl. Math., 33(1980), 507-544. 

\bibitem{Chrusciel2005} 
P. Chru\'sciel, E. Delay, J. Lee, D. Skinner, 
\emph{Boundary regularity of conformally compact 
Einstein metrics}, J. Diff. Geom., 69(2005), 
111-136.

\bibitem{Fefferman1976} C. Fefferman, 
\emph{Monge-Amp\`ere equation, the Bergman kernel, and geometry of pseudoconvex domains}, 
Ann. Math., 103(1976), 395-416. 

\bibitem{Fefferman&Graham2002} C. Fefferman, C. R. Graham, 
\emph{$Q$-curvature and Poincar\'e metrics}, Math. Res. Lett., 9(2002), 139-151. 

\bibitem{Fefferman&Graham2012} C. Fefferman, C. R. Graham, 
\emph{The Ambient Metric}, Annals of Mathematics Studies, 178, Princeton University Press, 
Princeton, 2012.

\bibitem{Friedman1958} A. Friedman, \emph{On the regularity of the solutions nonlinear elliptic and 
parabolic systems of partial differential equations}, J. Math. Mech., 7(1958), 43-59. 

\bibitem{GT1983} D. Gilbarg, N. Trudinger,
{\it Elliptic Partial Differential Equations of Elliptic Type},
Springer, Berlin, 1983.

\bibitem{Graham&Witten1999} C. R. Graham, E. Witten, 
\emph{Conformal anomaly of submanifold observables in AdS/CFT correspondence}, 
Nuclear Physics B, 546(1999), 52-64. 

\bibitem{Han2016} Q. Han, {\it Nonlinear Elliptic Equations of the Second Order}, Graduate Studies in Mathematics,
Volume 171,  Amer. Math. Soc., Providence, 2016.

\bibitem{HanJiang2014} Q. Han, X. Jiang,
\emph{Boundary expansions for minimal graphs in the hyperbolic space}, arxiv:1412.7608.

\bibitem{HanKhuri2014CPDE} Q. Han, M. Khuri, {\it Existence and blow-up behavior for solutions
of the generalized Jang equation}, Comm. P.D.E., 38(2013), 2199-2237. 


\bibitem{Han16CalVar} Q. Han, W. Shen, Y. Wang, {\it Optimal regularity of minimal graphs in the hyperbolic 
space}, Cal. Var. \& P. D. E., 55(2016), no. 1, Art. 3, 19pp. 

\bibitem{Hardt&Lin1987} R. Hardt, F.-H. Lin, \emph{Regularity at infinity 
for area-minimizing hypersurfaces in hyperbolic space}, Invent. Math., 88(1987), 217-224. 

%\bibitem{Kichenassamy2005JFA}
%S. Kichenassamy, \emph{Boundary behavior in the Loewner-Nirenberg problem}, 
%J. of Funct. Anal., 222(2005), 98-113. 

\bibitem{Hellimell2008} D. Helliwell, 
\emph{Boundary regularity for conformally compact Einstein metrics in 
even dimensions},
Comm. P.D.E., 33(2008),  
842-880.

\bibitem{JianWang2013JDG} H. Jian, X.-J. Wang, 
{\it Bernstein theorem and regularity for a class of Monge-Amp\`{e}re equations}, 
J. Diff. Geom., 93(2013), 431-469. 

\bibitem{JianWang2014Adv} H. Jian, X.-J. Wang, {\it Optimal boundary regularity
for nonlinear singular elliptic equations}, Adv. Math., 251(2014), 111-126. 

\bibitem{Kichenassamy2004Adv} 
{\it On a conjecture of Fefferman and Graham},
Adv. Math., 184(2004), 268-288.

\bibitem{Kichenassamy2005JFA}
S. Kichenassamy, \emph{Boundary behavior in the Loewner-Nirenberg problem}, 
J. of Funct. Anal., 222(2005), 98-113. 


\bibitem{KL:1} S. Kichenassamy, W. Littman,
{\it Blow-up surfaces for nonlinear wave equations, Part I},
Comm. P. D. E., 18(1993), 431-452. 


\bibitem{KL:2} S. Kichenassamy, W. Littman, 
{\it Blow-up surfaces for nonlinear wave equations, Part II},
Comm. P. D. E., 18(1993), 1869-1899.

\bibitem{LeeMelrose1982} J. Lee, R. Melrose, 
\emph{Boundary behavior of the complex Monge-Amp\`ere equation}, 
Acta Math., 148(1982), 159-192. 

\bibitem{Lin1989CPAM} F.-H. Lin, \emph{Asymptotic behavior of area-minimizing 
currents in hyperbolic space}, Comm. Pure Appl. Math., 42(1989), 229-242. 

\bibitem{Lin1989Invent} F.-H. Lin,
{\it On the Dirichlet problem for minimal graphs in hyperbolic space},
Invent. Math., 96(1989), 593-612.

\bibitem{Lin2012Invent} F.-H. Lin,
{\it Erratum: On the Dirichlet problem for minimal graphs in hyperbolic space},
Invent. Math., 187(2012), 755-757.

\bibitem{Loewner&Nirenberg1974} C. Loewner, L. Nirenberg, 
\emph{Partial differential equations invariant under conformal or projective transformations}, 
Contributions to Analysis, 245-272, Academic Press, New York, 1974. 

\bibitem{Mazzeo1988} R. Mazzeo, 
\emph{The Hodge cohomology of conformally compact metrics}, 
J. Diff. Geom., 28(1988), 309-339.

\bibitem{Mazzeo1991CPDE} R. Mazzeo, \emph{Elliptic theory of differential edge operators I}, 
Comm. P.D.E., 16(1991), 1615-1664. 

\bibitem{Mazzeo1991} R. Mazzeo, 
\emph{Regularity for the singular Yamabe problem}, Indiana Univ. Math. Journal, 40(1991), 1277-1299.

\bibitem{Mazzeo1991AJM} R. Mazzeo,
\emph{Unique continuation at infinity and embedded eigenvalues for asymptotically hyperbolic manifolds}, 
Amer. J. Math., 113(1991), 25-45.


\bibitem{Mazzeo&Melrose1987} R. Mazzeo, R. Melrose, 
\emph{Meromorphic extension of the resolvent on complete spaces with asymptotically 
negative curvature}, J. Funct. Anal., 75(1987), 260-310. 

\bibitem{Nirenberg1972} L. Nirenberg, \emph{An abstract form of the nonlinear
Cauchy-Kowalewski  theorem}, J. Diff. Geom., 6(1972), 561-579. 

\bibitem{Tonegawa1996MathZ} Y. Tonegawa, 
\emph{Existence and regularity of constant mean curvature hypersurfaces in hyperbolic space}, 
Math. Z., 221(1996), 591-615. 

\end{thebibliography}
\end{document}